\documentclass[letterpaper,11pt]{article}
\usepackage[T1]{fontenc}
\usepackage[english]{babel}
\usepackage{fullpage,graphicx,color,xcolor}
\usepackage{amsmath,amsfonts,amssymb,amsthm,dsfont,mathtools}
\usepackage{subcaption} 
\usepackage{algorithm,algpseudocode} 
\usepackage{cite}
\usepackage{hyperref}
\usepackage[capitalize,nameinlink,noabbrev]{cleveref}

\allowdisplaybreaks
\mathtoolsset{centercolon}
\hypersetup{colorlinks, hypertexnames=false, pageanchor=true,
            linkcolor=blue, citecolor=green!50!black, urlcolor=cyan,
            pdftitle={Accelerating gradient tracking with periodic global averaging}
            }
\newcommand{\reals}                  {\mathbb R}
\newcommand{\natint}                 {\mathbb N}
\newcommand{\ones}                   {\mathds 1}

\newcommand{\tran}                   {{\mathsf T}}                
\newcommand{\inprod}[2]              {\langle #1, #2 \rangle}     

\DeclareMathOperator{\Nullspace}     {Null}                       
\DeclareMathOperator{\linspan}       {{\bf span}}                 

\newcommand{\mini}                   {\text{\upshape minimize}}

\DeclareMathOperator{\dom}           {{\bf dom}}                  
\DeclareMathOperator{\intr}          {{\bf int}}                  

\renewcommand{\P}                    {\mathbb P}                  
\newcommand{\Prob}                   {\mathbb P}                  
\newcommand{\E}                      {\mathbb E}                  

\newcommand{\eg}{{\textit{e.g.}}}
\newcommand{\ie}{{\textit{i.e.}}}

\newcommand{\cE}{\mathcal{E}}
\newcommand{\cF}{\mathcal{F}}
\newcommand{\cG}{\mathcal{G}}

\newcommand{\cN}{\mathcal{N}}

\newcommand{\cV}{\mathcal{V}}

\newcommand{\bsf}{\boldsymbol{f}}
\newcommand{\bsg}{\boldsymbol{g}}

\newcommand{\bss}{\boldsymbol{s}}

\newcommand{\bsx}{\boldsymbol{x}}

\newcommand{\bsF}{\boldsymbol{F}}

\newcommand{\bsW}{\boldsymbol{W}}

\newcommand{\oline}[1]{\mkern 1.5mu\overline{\mkern-1.5mu#1}}

\renewcommand{\hbar}{\oline{h}}

\newcommand{\xbar}{\oline{x}}

\newcommand{\ftilde}{\tilde{f}}

\newcommand{\xtilde}{\tilde{x}}

\renewcommand{\mod}{\mathrm{mod}}
\renewcommand{\ones}{\mathds 1}

\newtheorem{assumption}{Assumption}
\newtheorem{theorem}{Theorem}
\newtheorem{lemma}[theorem]{Lemma}
\newtheorem{corollary}[theorem]{Corollary}

\crefname{assumption}{Assumption}{Assumptions}

\title{Accelerating gradient tracking with periodic global averaging%
\footnote{Shorter version is accepted to the 63rd IEEE Conference on Decision and Control (CDC 2024).}}
\author{Shujing Feng%
\thanks{Department of Computer Science and Engineering, Lehigh University, Email: \textsf{shf321@lehigh.edu}.}
\and Xin Jiang%
\thanks{Department of Industrial and Systems Engineering, Lehigh University, Email: \textsf{xjiang@lehigh.edu}.}
}
\date{March 17, 2024}

\begin{document}
\maketitle

\begin{abstract}
Decentralized optimization algorithms have recently attracted increasing attention due to its wide applications in all areas of science and engineering. In these algorithms, a collection of agents collaborate to minimize the average of a set of heterogeneous cost functions in a decentralized manner. State-of-the-art decentralized algorithms like Gradient Tracking (GT) and Exact Diffusion (ED) involve communication at each iteration. Yet, communication between agents is often expensive, resource intensive, and can be very slow. To this end, several strategies have been developed to balance between communication overhead and convergence rate of decentralized methods. In this paper, we introduce GT-PGA, which incorporates~GT with periodic global averaging. With the additional PGA, the influence of poor network connectivity in the GT algorithm can be compensated or controlled by a careful selection of the global averaging period. Under the stochastic, nonconvex setup, our analysis quantifies the crucial trade-off between the connectivity of network topology and the PGA period. Thus, with a suitable design of the PGA period, GT-PGA improves the convergence rate of vanilla GT. Numerical experiments are conducted to support our theory, and simulation results reveal that the proposed GT-PGA accelerates practical convergence, especially when the network is sparse.
\end{abstract}

\section{Introduction}

In decentralized optimization, a group of~$n$ agents collaborate to solve the optimization problem
\begin{equation} \label{eq:prob}
    \mini \quad f(x) := \frac{1}{n} \sum_{i=1}^{n} f_i(x),
\end{equation}
where the optimization variable is $x \in \reals^d$, and each component function $f_i(x)$ is smooth, potentially nonconvex, and held locally by agent $i \in [n]$. This problem formulation has been widely used in modeling various important applications throughout science and engineering, including optimal control, signal processing, resource allocation, and machine learning \cite{PKP06,BPC+11,CYRC12,NL18}. In particular, decentralized/distributed optimization is now prevalent in modern scenarios involving high-performance computing (HPC) resources \cite{YYH+21}.

Many decentralized methods have been proposed to solve the problem~\eqref{eq:prob}, including decentralized/distributed gradient descent (DGD) methods \cite{NO09,CS10,RNV10}, EXTRA \cite{SLWY15}, Exact-Diffusion/D$^2$/NIDS (ED) \cite{YYZS18,TLY+18,LSY19,YAYS20}, and Gradient Tracking (GT) methods \cite{XZSX15,DS16,NOS17,QL18}. Among them, DGD is arguably the conceptually simplest decentralized algorithms. At each iteration of DGD, each agent performs a local gradient step followed by a communication round. However, DGD fails to converge \textit{exactly} with constant stepsizes when the local objective functions $f_i$ are \textit{heterogeneous} \cite{CS13,YLY16} (\ie, the minimizer of $f$ is different from that of~$f_i$.).

Due to the unsatisfactory convergence results of DGD, exact methods (a.k.a.\ bias-correction methods) have been extensively studied to account for the inherent heterogeneity in problem~\eqref{eq:prob}. Among them, the family of GT algorithms have each agent perform local gradient steps with an estimate of the global gradient called the tracking variable \cite{XZSX15,DS16,NOS17,QL18}. In these methods, the bias (or error) caused by problem/data heterogeneity observed in DGD is asymptotically removed.

In decentralized methods (including both DGD and exact methods), gossip communication over the network of agents is required at each iteration of the algorithm. Very often, communication is computationally expensive and resource intensive in practice \cite{NOR18,YYH+21}. To this end, multiple local recursions (or local updates) have recently been studied in the literature. Among these methods include LocalGD \cite{Stich18,KMR20}, Scaffold \cite{KKM+20}, S-Local-GD \cite{GHR21}, FedLin \cite{MJPH21}, and Scaffnew \cite{MMSR22}. LocalGD, which is based on DGD, still suffers from the bias caused by heterogeneity, and multiple local recursions cause agents to drift towards their local solution \cite{KLB+20}. Other aforementioned methods combine bias-correction techniques with multiple local gradient updates; for example, GT with local updates (LU-GT) have recently been studied \cite{BBG23a,GC23,LLKS23,NAYU23}. Nonetheless, existing analyses fail to establish any theoretical improvement in communication complexity when the number of local updates are deterministically prescribed \cite{Alghunaim24,BBG23a,GC23,LLKS23,NAYU23}. To the best of our knowledge, the only existing analysis that theoretically establishes the benefit of local updates in LocalGD \cite{MMSR22} considers the special case where the objective is strongly convex, the true gradients $\nabla f_i$ are always accessible, and the number of local updates is randomly selected during the optimization process.

Besides local updates, the \textit{periodic global averaging} (PGA) technique has recently been developed \cite{CYZ+21} to balance the trade-off between convergence and communication in DGD. It is shown that PGA helps improve the transient stage of DGD with and without local updates \cite{CYZ+21}. In modern scenarios where high-performance data-center clusters are the computing resources, PGA is beneficial owing to efficient All-Reduce primitives \cite{PY09}. In addition, the benefit of PGA in DGD is significant when the network is large and/or sparse. However, PGA does not remove the heterogeneity bias in DGD, so DGD with PGA still does not converge exactly with constant stepsizes.

In view of the potential benefits of PGA and the undesirable performance of DGD-PGA, in this work, we incorporate periodic global averaging (PGA) into GT and propose GT-PGA. On the one hand, we show that the incorporation of PGA accelerates the convergence rate compared with vanilla GT, especially on large and/or sparse networks. On the other hand, GT-PGA also extends LU-GT (with fully connected networks) via efficient gossip communication after local updates.

Despite the promising acceleration in practical convergence, the analysis of GT-PGA is not straightforward. Even though the main recursion of GT-PGA can be regarded as a special form of GT with time-varying topologies \cite{DS16}, its convergence guarantees and practical performance cannot be fully captured by existing analyses. In particular, existing convergence results for time-varying GT rely on the spectral gap of the least connected communication network \cite{DS16}. Simply applying these results to GT-PGA does not fully explain the superiority of the PGA operation and lead to incomplete conclusions. Therefore, quantifying the benefits of PGA in GT and carefully balancing the trade-off between gossip communication and periodic global averaging require new analysis of the decentralized algorithm. 

Overall, the contributions of this paper are as follows.
\begin{itemize}
    \item We propose to incorporate periodic global averaging (PGA) into the Gradient Tracking (GT) algorithm and analyze the proposed GT-PGA under the stochastic, nonconvex setting. 
    
    \item Theoretical results are established to guarantee convergence of GT-PGA, and in particular, to quantify the crucial trade-off between network connectivity and the global averaging period. We also discuss the connection and difference between the proposed GT-PGA, vanilla GT, and LU-GT (GT with local updates).
    
    \item Numerical experiments are conducted to verify the established theoretical results. In particular, the proposed GT-PGA accelerates practical convergence compared to vanilla GT, especially when the network is large and/or sparse.
\end{itemize}
The rest of the paper is organized as follows. \cref{sec:algo} describes the proposed GT-PGA algorithm and presents the main convergence results. In \cref{sec:analysis}, we establish the convergence guarantees for GT-PGA, under the stochastic, nonconvex setting, and \cref{sec:exp} presents numerical evidence to support the theoretical results. Finally, \cref{sec:conclusion} presents concluding remarks.

\paragraph{Notation.}
Lowercase letters define vectors or scalars, uppercase letters define matrices or scalars, and boldface letters represent augmented network quantities. Let $\mathrm{col}\{a_1,\ldots,a_n\}$ denote the vector that concatenates the vectors/scalars $a_i$, and define $[n] := \{1,\ldots,n\}$ for any positive integer $n \in \natint_{\geq 1}$. The notation $\ones$ represents the all-ones vector, of which the size will be clear from context. The inner product of two vectors $x$, $y$ is denoted by $\inprod{x}{y}$. For any real $p \times q$ matrix $A$, denote its nullspace by $\Nullspace(A) := \{x \in \reals^q \mid Ax=0\}$. Products of multiple matrices are defined as
\[
    \prod_{k=i}^j A_k := \begin{cases}
        A_i A_{i+1} \cdots A_j \quad &\text{if} \ j \geq i \\
        A_i A_{i-1} \cdots A_j \quad &\text{if} \ j < i.
    \end{cases}
\]
Note that we do not assume $j \geq i$. This definition will not cause any confusion as the value of $i$ and $j$ will be clear from context.

\section{Gradient tracking with periodic global averaging} \label{sec:algo}

In this section, we present the proposed decentralized optimization algorithm for solving problem~\eqref{eq:prob}, state the assumptions needed in the analysis, and establish the main convergence result for the proposed algorithm.

In problem~\eqref{eq:prob}, the function $f_i \colon \reals^d \to \reals$ held locally by agent $i$ is smooth, potentially nonconvex, and defined as the expected value with respect to some probability space $(\Omega_i, \cF_i, \Prob_i)$; \ie,
\[
    f_i(x) := \E_{\xi_i} [F_i(x;\xi_i)], \quad \text{for all} \ i \in [n].
\]
The studied optimization algorithm only has access to stochastic gradient estimates of the true gradient of $f_i$ (see upcoming \cref{ass:noise}), and solves the problem~\eqref{eq:prob} in a decentralized manner. Its implementation involves a graph $\cG = (\cV, W, \cE)$ that models the connections between the group of~$n$ agents (\ie, $|\cV|=n$). Specifically, the element $w_{ij}$ in the matrix $W$ scales the information agent $i$ receives from agent~$j$, and $w_{ij} = 0$ if $j \notin \cN_i$, where $\cN_i$ is the set of neighbors of agent~$i$.

In this work, we incorporate periodic global averaging into the well-known Gradient Tracking (GT) algorithms \cite{DS16,NOS17} and study its convergence results. Various forms of GT exist in the literature, and the particular variant of GT considered in this paper is called Semi-ATC-TV-GT \cite{DS16}. The proposed algorithm, called Gradient Tracking with Periodic Global Averaging (GT-PGA), is listed in \cref{alg:gt}. In the gossip (\ie, decentralized communication) steps (\cref{alg:gossip} in \cref{alg:gt}), every agent~$i$ collects information from all its \textit{connected} neighbors, while for global averaging steps (\cref{alg:avg} in \cref{alg:gt}), agents synchronize their model parameters using, \eg, the efficient All-Reduce primitives \cite{PY09}. When the global averaging period $\tau \to \infty$, GT-PGA reduces to Gradient Tracking \cite{DS16} with static topology; when $W=I$, GT-PGA reduces to GT with local updates and a fully-connected network~\cite{NAYU23}.

\begin{algorithm}[tb]
\caption{Gradient Tracking with Periodic Global Averaging (GT-PGA)}
\label{alg:gt}
\begin{algorithmic}[1]
    \State Agent $i$ input: $x_i^{(0)} \in \reals^d$, stepsize $\alpha \in \reals_{>0}$, mixing matrix $W \in \reals^{n \times n}$, averaging period $\tau \in \natint_{\geq 1}$.
    \State Initialize $g_i^{(0)} = \nabla F_i(x_i^{(0)}, \xi_i^{(0)}) \in \reals^d$ for all $i \in [n]$.
    \For{$k=0,1,\ldots$}
        \For{$i=1,\ldots, n$ (in parallel)}
            \If{$\mod(k+1,\tau) = 0$} \label{alg:avg}
                \begin{align*}
                    x_i^{(k+1)} &= \frac{1}{n} \sum_{j=1}^n (x_j^{(k)} - \alpha g_j^{(k)}) \\
                    g_i^{(k+1)} &= \frac{1}{n} \sum_{j=1}^n g_j^{(k)} + \nabla F_i(x_i^{(k+1)}, \xi_i^{(k+1)}) - \nabla F_i(x_i^{(k)}, \xi_i^{(k)}).
                \end{align*}
            \Else \label{alg:gossip}
                \begin{align*}
                    x_i^{(k+1)} &= \sum_{j : (j,i)\in \cE}^{\phantom{=}} w_{ij} (x_j^{(k)} - \alpha g_j^{(k)}) \\
                    g_i^{(k+1)} &= \sum_{j : (j,i)\in \cE} w_{ij} g_j^{(k)} + \nabla F_i(x_i^{(k+1)}, \xi_i^{(k+1)}) - \nabla F_i(x_i^{(k)}, \xi_i^{(k)}).
                \end{align*}
            \EndIf
        \EndFor
    \EndFor
\end{algorithmic}
\end{algorithm}

The proposed periodic global averaging technique is efficient in situations where high-performance data-center clusters are the computing resources. In such a scenario, all GPUs are fully connected with high-bandwidth channels and the network topology can be fully controlled. Under this setting, PGA conducted with Ring All-Reduce has tolerable communication cost; see, \eg, \cite{CYZ+21}. For scenarios where PGA is extremely expensive (\eg, in wireless sensor networks), PGA can be approximated via multiple gossip steps, or may not be recommended.

To write \cref{alg:gt} in a more concise form, we introduce the network notation:
\begin{align*}
    \bsx^{(k)} &:= \mathrm{col} \{x^{(k)}_1, \ldots, x^{(k)}_n\} \in \reals^{nd}, \\
    \bsg^{(k)} &:= \mathrm{col} \{g^{(k)}_1, \ldots, g^{(k)}_n\} \in \reals^{nd}, \\
    \nabla \bsf^{(k)} &:= \mathrm{col} \{\nabla f_1(x^{(k)}), \ldots, \nabla f_n(x^{(k)})\} \in \reals^{nd}, \\
    \nabla \bsF^{(k)} &:= \mathrm{col} \{\nabla F_1(x_1^{(k)}; \xi_1^{(k)}), \ldots, \nabla F_n(x_n^{(k)}; \xi_n^{(k)})\}, \\
    \hat \bsx^{(k)} &:= \bsx^{(k)} - \ones_n \otimes \xbar^{(k)} \in \reals^{nd}, \\
    \bsW &:= W \otimes I_d, \;\; \widehat\bsW := \bsW - \frac{1}{n} \ones_n \ones_n^\tran \otimes I_d, \\
    \bsf(\bsx^{(k)}) &:= \frac{1}{n} \sum\limits_{i=1}^{n} f_i(x), \;\; \overline{\nabla f} (\bsx^{(k)}) := \frac{1}{n} \sum\limits_{i=1}^n \nabla f_i(x_i^{(k)}), \\ 
    \xbar^{(k)} &:= \frac{1}{n} \sum\limits_{i=1}^n x_i^{(k)}.
\end{align*}
With the augmented notations, the main recursion of \cref{alg:gt} can be written concisely as:
\begin{subequations} \label{eq:gt-mat}
\begin{align*}
    \bsx^{(k+1)} &= \bsW^{(k)} (\bsx^{(k)} - \alpha \bsg^{k}) \\
    \bsg^{(k+1)} &= \bsW^{(k)} \bsg^{(k)} + \nabla \bsF (\bsx^{(k+1)}; \boldsymbol\xi^{(k+1)}) - \nabla \bsF (\bsx^{(k)}; \boldsymbol\xi^{(k)}),
\end{align*}
\end{subequations}
where $\bsW^{(k)} := W^{(k)} \otimes I_d$ and
\[
    W^{(k)} = \begin{cases}
        \frac{1}{n} \ones \ones^\tran \quad &\text{if $\mod(k+1, \tau) = 0$} \\
        W &\text{otherwise.} 
    \end{cases}
\]

Now, we list all the assumptions needed for the analysis.
\begin{assumption}[Mixing matrix] \label{ass:mixing}
    The network is strongly connected, and the mixing matrix $W \in \reals^{n \times n}$ satisfies $W \ones = \ones$, $W^\tran \ones = \ones$, and $\Nullspace(I - W) = \linspan(\ones)$. In addition, denote
    \[
        \beta := \|W - \tfrac{1}{n} \ones \ones^\tran\|_2 \in (0,1).
    \]
\end{assumption}
The quantity $\beta$ indicates how well the network is connected. A smaller $\beta$ indicates a better connected network while a larger one implies a worse connectivity. 

The following two assumptions are made on the problem~\eqref{eq:prob}. In particular, convexity is not assumed, and the algorithm only has access to stochastic gradient estimates of each local function.
\begin{assumption}[$L$-smoothness] \label{ass:smooth}
    Each function $f_i \colon \reals^d \to \reals$ is continuously differentiable with an $L$-Lipschitz continuous gradient; \ie, there exists a constant $L \in \reals_{>0}$ such that
    \[
        \|\nabla f_i(x) - \nabla f_i(y)\| \leq L \|x-y\|,
    \]
    for all $(x, y) \in \intr \dom f_i \times \intr \dom f_i$ and for all $i \in [n]$. In addition, the objective function $f \colon \reals^d \to \reals$ is bounded below, and the optimal value of problem~\eqref{eq:prob} is denoted by $f^\star \in \reals$.
\end{assumption}
At iteration~$k$ of \cref{alg:gt}, a stochastic gradient estimator of each component function~$f_i$ is computed, based on the random variable $\xi_i^{(k)} \in (\Omega_i, \cF_i, \P_i)$. let~$\cF^{(0)}$ denote the $\sigma$-algebra corresponding to the initial conditions and, for all $k \in \natint_{\geq 1}$, let $\cF^{(k)}$ denote the $\sigma$-algebra defined by~$\{\bsx^{(j)}\}_{j=0}^k$. The following assumption is made on the stochastic gradient estimator.
\begin{assumption}[Stochastic noise] \label{ass:noise}
    For all $k \in \natint$ and for all $i \in [n]$, the random variables $\xi_i^{(k)}$ are independent of each other.
    The stochastic gradient estimator satisfies
    \[
        \E [\nabla F_i(x_i^{(k)}; \xi_i^{(k)}) \mid \cF^{(k)}] = \nabla f_i(x_i^{(k)}), \quad \text{for all} \ i \in [n].
    \]
    In addition, there exists $\sigma \in \reals_{>0}$ such that for all $k \in \natint$ and for all $i \in [n]$, it holds that
    \[
        \E [\|\nabla F_i(x_i^{(k)}; \xi_i^{(k)}) - \nabla f_i (x_i^{(k)})\|^2 \mid \cF^{(k)}] \leq \sigma^2.
    \]
\end{assumption}

We now state the main result of this paper on the convergence guarantees of \cref{alg:gt}.
\begin{theorem}[Convergence of GT-PGA] \label{thm:conv}
    Let \cref{ass:mixing,ass:smooth,ass:noise} hold, let $\tau \in \natint_{\geq 2}$, and let the stepsize satisfy $\alpha \leq \min \big\{\frac{1}{2L}, \frac{1}{4\sqrt{6} \beta \tau^2 L} \big\}$. Then, for any $K \in \natint_{\geq \tau+1}$, the sequence $\{\bsx^{(k)}\}$ generated by \cref{alg:gt} satisfies
    \begin{equation} \label{eq:conv-thm}
        \frac{1}{K+1} \sum\limits_{k=0}^K \Big(\E \| \overline{\nabla f} (\bsx^{(k)})\|^2 + \E \|\nabla f(\xbar^{(k)})\|^2 \Big) \leq \frac{\gamma_1 L^2}{nK} + \frac{\gamma_2 \beta \tau^2 L^2}{K} + \gamma_3 \sigma^2 \left(\frac{1}{(1-\beta^2) \tau^2} + \frac{1}{\beta \tau^2 n} \right)
    \end{equation} 
    with some constants $(\gamma_1, \gamma_2, \gamma_3) \in \reals_{>0} \times \reals_{>0} \times \reals_{>0}$.
\end{theorem}
In general, GT-PGA exhibits an $O(1/K)$ convergence rate, consistent with the results for gradient tracking algorithms under the stochastic, nonconvex setup (see, \eg, \cite{DS16}). As is typical in the literature, the first term on the right-hand side of~\eqref{eq:conv-thm} is related to the number of agents and independent of the topology as well as the PGA period $\tau$. The crucial trade-off between the connectivity of the communication network ($\beta$) and the PGA period ($\tau$) is depicted in the second term in~\eqref{eq:conv-thm}.
\begin{itemize}
    \item When the network is large or sparse (\ie, $\beta \to 1$), global averaging is more critical to drive consensus and a smaller $\tau$ is needed to compensate the negative effect of poor connectivity.

    \item When the network is small or dense, gossip communication is already helpful enough to achieve consensus and a larger $\tau$ can be used. In the extreme case, GT-PGA with $\tau \to \infty$ reduces to vanilla GT.

    \item Recall that when $W=I$, GT-PGA reduces to GT with fully-connected graphs and with local updates (LU-GT). Thus, gossip communication in GT-PGA also contributes to consensus, and this property is critical to establish the superiority of GT-PGA with LU-GT.
\end{itemize}
Therefore, thanks to the periodic global averaging operation, GT-PGA enjoys promising convergence properties compared with vanilla GT and LU-GT, and our analysis supports the above discussion.

We end this section with an auxiliary convergence result with further tuning on the stepsize $\alpha$. The same stepsize tuning strategy is common in the literature on decentralized optimization; see, \eg, \cite{Stich18,KKM+20,KLB+20,Alghunaim24}.
\begin{corollary} \label{cor:conv}
    Let \cref{ass:mixing,ass:smooth,ass:noise} hold, let $\tau \in \natint_{\geq 2}$, and let the stepsize satisfy 
    \begin{equation} \label{eq:cor-alpha}
        \alpha = \min \left\{\left(\frac{nL}{K \sigma^2} \right)^{\frac{1}{2}}, \left(\frac{1-\beta^2}{\beta^2 \tau^2 K \sigma^2} \right)^{\frac{1}{2}}, \frac{1}{2L}, \frac{1}{4\sqrt{6} \beta \tau^2 L}\right\}.
    \end{equation} 
    In addition, suppose $n \gg 1/(\beta \tau)^2$ (\eg, the number of agents~$n$ is sufficiently large or the network is sufficiently sparse). Then, for any $K \in \natint_{\geq \tau+1}$, the sequence $\{\bsx^{(k)}\}$ generated by \cref{alg:gt} satisfies
    \[
        \frac{1}{K+1} \sum\limits_{k=0}^K \Big(\E \| \overline{\nabla f} (\bsx^{(k)})\|^2 + \E \|\nabla f(\xbar^{(k)})\|^2 \Big) \leq \frac{\gamma_4 \beta \tau^2 L^3}{K} + \frac{\gamma_5 L^{\frac{3}{2}} \sigma}{(nK)^{\frac{1}{2}}} + \frac{\gamma_6 \beta \tau L^2 \sigma}{(1-\beta^2)^{\frac{1}{2}} K^{\frac{1}{2}}}
    \]
    with some constants $(\gamma_4, \gamma_5, \gamma_6) \in \reals_{>0} \times \reals_{>0} \times \reals_{>0}$.
\end{corollary}

\section{Algorithm analysis} \label{sec:analysis}

This section presents the theoretical analysis of \cref{alg:gt} stated in \cref{thm:conv}. As typical in the analyses of decentralized algorithms, the two important pillars are the \textit{descent inequality} and the \textit{consensus inequality}. The descent inequality establishes the convergence properties of the averaged iterates $\xbar^{(k)}$ to a first-order stationary point and is standard in the analyses of GT (see upcoming \cref{lem:descent}). The consensus inequality is different from existing analyses and characterizes the per-iteration behavior of the consensus error; see upcoming \cref{lem:consensus}.
\begin{lemma}[Descent inequality \textnormal{\cite[Lemma~5.1]{NJYU23}}] \label{lem:descent}
    Let \cref{ass:mixing,ass:smooth,ass:noise} hold, and let the stepsize satisfy $\alpha \in \big(0,\tfrac{1}{2L} \big]$. Denote $\ftilde := f - f^\star$. Then, the sequence generated by \cref{alg:gt} satisfies
    \[
        \frac{1}{K+1} \sum\limits_{k=0}^K \big( \E \|\overline{\nabla f} (\bsx^{(k)})\|^2 + \E \|\nabla f(\xbar^{(k)})\|^2 \big) \leq \frac{4}{\alpha (K+1)} \E \ftilde (\xbar^{(0)}) + \frac{2L^2}{n(K+1)} \sum_{k=0}^K \E \|\hat \bsx^{(k)}\|^2 + \frac{2\alpha L \sigma^2}{n},
    \]
    for all $k \in \natint$.
\end{lemma}
Note that this inequality does not involve the mixing matrix $W$, so it holds for gradient tracking with static topology as well as the proposed GT-PGA. Its derivation is standard in the literature and thus omitted here.

The second lemma studies the behavior of the consensus error and is used to establish that all agents' local variables converge to their average.
\begin{lemma}[Consensus inequality] \label{lem:consensus}
    Let \cref{ass:mixing,ass:smooth,ass:noise} hold, let $\tau \in \natint_{\geq 2}$, and let the stepsizes satisfy $\alpha \in \big(0, \tfrac{1}{4\sqrt{6} \beta \tau^2 L} \big]$. Then, for $K \in \natint_{\geq \tau+1}$, the iterates generated by \cref{alg:gt} satisfy
    \begin{align}
        \frac{1}{K+1} \sum_{k=0}^K \E \|\hat \bsx^{(k)}\|^2 &\leq \frac{2}{K+1} \sum_{k=0}^\tau \E [\|\hat{\bsx}^{(k)}\|^2] + \frac{n}{192 \beta^2 \tau^4 L^2 (K+1)} \sum_{k=0}^K \E [\|\nabla f(\xbar^{(k)})\|^2] \nonumber \\ 
        &\phantom{\leq} \mathrel{+} \left(\frac{1}{768 \beta^2 \tau^4 L^2} + \frac{n}{24\tau^4 L^2} + \frac{n}{6 (1-\beta^2) \tau^2 L^2}\right) \sigma^2. \label{eq:lem-consensus}
    \end{align} 
\end{lemma}
\begin{proof}
    The analysis relies on the following reformulation of~\eqref{eq:gt-mat}:
    \[
        \begin{bmatrix} \bsx^{(k+1)} \\ \bsg^{(k+1)} \end{bmatrix} = 
        \begin{bmatrix} \bsW^{(k)} & -\alpha \bsW^{(k)} \\ 0 & \bsW^{(k)} \end{bmatrix} 
        \begin{bmatrix} \bsx^{(k)} \\ \bsg^{(k)} \end{bmatrix} + 
        \begin{bmatrix} 0 \\ \nabla \bsF(\bsx^{(k+1)}; \boldsymbol \xi^{(k+1)}) - \nabla \bsF(\bsx^{(k)}; \boldsymbol \xi^{(k)}) \end{bmatrix}.
    \]
    The same reformulation has been used in the literature \cite{SLJ+22,NJYU23}. Note that the coefficient matrix on the right-hand side is block upper-triangular, and thus
    \[
        \begin{bmatrix} \bsW^{(i)} & -\alpha \bsW^{(i)} \\ 0 & \bsW^{(i)} \end{bmatrix} \cdots
        \begin{bmatrix} \bsW^{(j)} & -\alpha \bsW^{(j)} \\ 0 & \bsW^{(j)} \end{bmatrix}
        = \begin{bmatrix} \prod_{k=i}^j \bsW^{(k)} & -\alpha (i-j+1) \prod_{k=i}^j \bsW^{(k)} \\ 0 & \prod_{k=i}^j \bsW^{(k)} \end{bmatrix}.
    \]
    It follows from the reformulation and the above identity that when $k-1 \geq \tau$, it holds that
    \begin{equation}
        \bsx^{(k)} = \Big(\prod_{i=k-1}^{0} \bsW^{(i)}\Big) \bsx^{(0)} - \alpha \sum_{j = 0}^{k-1} (k-j) \Big(\prod_{i=k-1}^{j} \bsW^{(i)} \Big) \big(\nabla \bsF(\bsx^{(j)}, \boldsymbol\xi^{(j)}) - \nabla \bsF (\bsx^{(j-1)}, \boldsymbol\xi^{(j-1)}) \big), \label{eq:lem-consensus-x}
    \end{equation}
    where we set $\nabla \bsF (\bsx^{(-1)}, \boldsymbol\xi^{(-1)}) \equiv \nabla \bsf (\bsx^{(-1)}) := 0$. It is assumed that $k-1 \geq \tau$, and thus at least one matrix $W^{(k)}$ (resp., $\bsW^{(k)}$) is the scaled all-ones matrix $\tfrac{1}{n} \ones_n \ones_n^\tran$ (resp., $\tfrac{1}{n} \ones_n \ones_n^\tran \otimes I_d$). Then, \eqref{eq:lem-consensus-x} can be rewritten as
    \begin{align*}
        \bsx^{(k)} &= \big(\prod\limits_{i=k-1}^{0} \bsW^{(i)}\Big) \bsx^{(0)} - \alpha \sum\limits_{j = 0}^{k-1-\tau} (k-j) \Big(\prod\limits_{i=k-1}^{j} \bsW^{(i)} \Big) \big(\nabla \bsF(\bsx^{(j)}, \boldsymbol\xi^{(j)}) - \nabla \bsF(\bsx^{(j-1)}, \boldsymbol\xi^{(j-1)}) \big) \\
        &\phantom{=} \mathrel{-} \alpha \sum\limits_{j = k - \tau}^{k-1} (k-j) \Big(\prod\limits_{i=k-1}^{j} \bsW^{(i)} \Big) \big(\nabla \bsF(\bsx^{(j)}, \boldsymbol\xi^{(j)}) - \nabla \bsF(\bsx^{(j-1)}, \boldsymbol\xi^{(j-1)}) \big),
    \end{align*}
    where we split the summation for later convenience. Then, we multiply both sides of the equation by $I_{nd} - \tfrac{1}{n} \ones_n \ones_n^\tran \otimes I_d$, denote $\bss^{(k)} := \nabla \bsF(\bsx^{(k)}; \boldsymbol \xi^{(k)}) - \nabla \bsf(\bsx^{(k)}) \in \reals^{nd}$ (with $\bss^{(0)} := 0$), and obtain
    \begin{align*}
        \hat\bsx^{(k)} &= \Big(\prod_{i=k-1}^{0} \widehat\bsW^{(i)}\Big) \hat{\bsx}^{(0)} - \alpha \sum_{j = 0}^{k-1-\tau} (k-j) \Big(\prod_{i=k-1}^{j} \widehat\bsW^{(i)} \Big) \big(\nabla \bsF (\bsx^{(j)}, \boldsymbol\xi^{(j)}) - \nabla \bsF (\bsx^{(j-1)}, \boldsymbol\xi^{(j-1)}) \big) \\
        &\phantom{\mathrel{=}} \; - \alpha \sum_{j = k - \tau}^{k-1} (k-j) \Big(\prod_{i=k-1}^{j} \widehat\bsW^{(i)} \Big) \big(\nabla \bsF (\bsx^{(j)}, \boldsymbol\xi^{(j)}) - \nabla \bsF (\bsx^{(j-1)}, \boldsymbol\xi^{(j-1)}) \big) \nonumber \\ 
        &= -\alpha \sum_{j = k - \tau}^{k-1} (k-j) \Big(\prod_{i=k-1}^{j} \widehat\bsW^{(i)} \Big) \big(\nabla \bsF (\bsx^{(j)}, \boldsymbol\xi^{(j)}) - \nabla \bsF (\bsx^{(j-1)}, \boldsymbol\xi^{(j-1)}) \big) \\ 
        &= - \alpha \sum_{j = k - \tau}^{k-1} (k-j) \Big(\prod_{i=k-1}^{j} \widehat{\bsW}^{(i)} \Big) \big(\nabla \bsf(\bsx^{(j)}) - \nabla \bsf (\bsx^{(j-1)}) \big) \\ 
        &\phantom{=} \; - \alpha \sum_{j = k - \tau}^{k-1} (k-j) \Big(\prod_{i=k-1}^{j} \widehat{\bsW}^{(i)} \Big) \big(\bss^{(j)} - \bss^{(j-1)} \big) \\ 
        &= - \alpha \sum_{j = k - \tau}^{k-1} (k-j) \Big(\prod_{i=k-1}^{j} \widehat{\bsW}^{(i)} \Big) \big(\nabla \bsf(\bsx^{(j)}) - \nabla \bsf (\bsx^{(j-1)}) \big) - \alpha \widehat{\bsW}^{(k-1)} \bss^{(k-1)} \\
        &\phantom{=} \; - \alpha \sum_{j=k-\tau}^{k-2} \Big( (k-j) \Big( \prod_{i=k-1}^{j} \widehat{\bsW}^{(i)} \Big) - (k-j-1) \Big( \prod_{i=k-1}^{j+1} \widehat{\bsW}^{(i)} \Big) \Big) \bss^{(j)}.
    \end{align*}
    where the second equality uses the fact that $W^{(k)} = \tfrac{1}{n} \ones \ones^\tran$ and $\widehat\bsW^{(k)} = 0$ if $\mod(k+1,\tau) = 0$. Then, the expectation of $\|\hat\bsx^{(k)}\|^2$ conditioned on $\cF^{(k)}$ is bounded by 
    \begin{align}
        \MoveEqLeft[0.2] \E_k \big[\big\lVert \hat{\bsx}^{(k)} \big\rVert^2 \big] \nonumber \\ 
        &\leq 2 \E_k \Big[\Big\lVert \alpha \sum_{j = k - \tau}^{k-1} (k-j) \Big(\prod_{i=k-1}^{j} \widehat{\bsW}^{(i)} \Big) \big(\nabla \bsf(\bsx^{(j)}) - \nabla \bsf (\bsx^{(j-1)}) \big) \Big\rVert^2\Big] \nonumber \\ 
        &\phantom{\leq} \; \mathrel{+} 2 \E_k \Big[\Big\lVert \alpha \widehat{\bsW}^{(k-1)} \bss^{(k-1)} + \alpha \sum_{j=k - \tau}^{k-2} \Big((k-j) \Big( \prod_{i=k-1}^{j} \widehat{\bsW}^{(i)} \Big) + (k-j-1) \Big( \prod_{i=k-1}^{j+1} \widehat{\bsW}^{(i)} \Big) \Big) \bss^{(j)} \Big\rVert^2 \Big], \label{eq:lem-consensus-xhat-norm}
    \end{align}
    where we denote $\E_k := \E[\cdot \vert \cF^{(k)}]$ and apply Jensen's inequality. We then bound the two terms on the right-hand side of~\eqref{eq:lem-consensus-xhat-norm} one by one. The first term can be bounded by
    \begin{align*}
        \MoveEqLeft[0.2] \E_k \Big[\Big\lVert \alpha \sum_{j = k - \tau}^{k-1} (k-j) \Big(\prod_{i=k-1}^{j} \widehat{\bsW}^{(i)} \Big) \big(\nabla \bsf(\bsx^{(j)}) - \nabla \bsf (\bsx^{(j-1)}) \big) \Big\rVert^2\Big] \nonumber \\ 
        &\leq \alpha^2 \tau^3 \sum_{j = k - \tau}^{k-1} \E_k \Big[\Big\lVert \Big(\prod_{i=k-1}^{j} \widehat{\bsW}^{(i)} \Big) \big(\nabla \bsf(\bsx^{(j)}) - \nabla \bsf (\bsx^{(j-1)}) \big) \Big\rVert^2\Big] \\ 
        &\leq \alpha^2 \tau^3 \sum_{j = k - \tau}^{k-1} \beta^{2(k-j)} \Big(6\alpha^2 n L^2 \E_k [\|\nabla f(\xbar^{(j-1)})\|^2] + 9L^2 \E_k [\|\hat \bsx^{(j-1)}\|^2] + 3L^2 \E_k [\|\hat \bsx^{(j)}\|^2] + 3\alpha^2 L^2 \sigma^2 \Big) \\
        &\leq \alpha^2 \beta^2 \tau^3 \sum_{j = k - \tau}^{k-1} \Big(6\alpha^2 n L^2 \E_k [\|\nabla f(\xbar^{(j-1)})\|^2] + 9L^2 \E_k [\|\hat \bsx^{(j-1)}\|^2] + 3L^2 \E_k [\|\hat \bsx^{(j)}\|^2] + 3\alpha^2 L^2 \sigma^2 \Big).
    \end{align*}
    The first inequality follows from Jensen's inequality. In the second inequality, we use the \cref{ass:mixing}, the property of matrix norms, and \cite[Lemma~11]{SLJ+22}. Then in the last inequality, we use the fact that $\beta \in (0,1)$ and take the maximum of $\beta^{2(k-j)}$ over $j=k-\tau, \ldots, k-1$.

    Similarly, the second term on the right-hand side of~\eqref{eq:lem-consensus-xhat-norm} is bounded by
    \begin{subequations}
    \begin{align}
        \MoveEqLeft[0.2] \E_k \Big[\Big\lVert \alpha \widehat{\bsW}^{(k-1)} \bss^{(k-1)} + \alpha \sum_{j=k - \tau}^{k-2} \Big((k-j) \Big( \prod_{i=k-1}^{j} \widehat{\bsW}^{(i)} \Big) + (k-j-1) \Big( \prod_{i=k-1}^{j+1} \widehat{\bsW}^{(i)} \Big) \Big) \bss^{(j)} \Big\rVert^2 \Big] \nonumber \\ 
        &= \alpha^2 \E_k [\|\widehat\bsW^{(k-1)} \bss^{(k-1)}\|^2] \nonumber \\ 
        &\phantom{=} \mathrel{+} \alpha^2 \sum_{k-\tau}^{k-2} \E_k \Big[\Big\lVert \Big( (k-j) \Big(\prod_{i=k-1}^j \widehat\bsW^{(i)} \Big) - (k-j-1) \Big( \prod_{i=k-1}^{j+1} \widehat\bsW^{(i)} \Big) \Big) \bss^{(j)} \Big\rVert^2 \Big] \label{eq:lem-consensus-prf-1a} \\ 
        &\leq \alpha^2 \beta^2 n \sigma^2 + 2\alpha^2 \sum_{j=k-\tau}^{k-2} \E_k \Big[ \Big( (k-j)^2 \Big\lVert \prod_{i=k-1}^j \widehat\bsW^{(j)} \Big\rVert_2^2 + (k-j-1)^2 \Big\lVert \prod_{i=k-1}^{j+1} \widehat\bsW^{(j)} \Big\rVert_2^2 \Big) \|\bss^{(j)}\|^2 \Big] \label{eq:lem-consensus-prf-1b} \\ 
        &\leq \alpha^2 \beta^2 n \sigma^2 + 4\alpha^2 \tau^2 \sum_{j=k-\tau}^{k-2} \beta^{2(k-j-1)} n \sigma^2 \label{eq:lem-consensus-prf-1c} \\ 
        &\leq \alpha^2 \beta^2 n \sigma^2 + 4\alpha^2 \tau^2 n \sigma^2 \frac{\beta^2 - \beta^{2\tau}}{1-\beta^2} \label{eq:lem-consensus-prf-1d} \\ 
        &\leq \alpha^2 \beta^2 n \sigma^2 \Big(1 + \frac{4\tau^2}{1-\beta^2} \Big). \label{eq:lem-consensus-prf-1e}
    \end{align}
    \end{subequations}
    In~\eqref{eq:lem-consensus-prf-1a} we use the independence of the gradient noise. Then,~\eqref{eq:lem-consensus-prf-1b} and~\eqref{eq:lem-consensus-prf-1c} use \cref{ass:mixing,ass:noise}. In~\eqref{eq:lem-consensus-prf-1d} we calculate the sum of the geometric series, and~\eqref{eq:lem-consensus-prf-1e} uses $\beta \in (0,1)$.

    Substituting the two bounds back to~\eqref{eq:lem-consensus-xhat-norm} gives
    \begin{align}
        \E_k \big[\big\lVert \hat{\bsx}^{(k)} \big\rVert^2 \big] &\leq 2\alpha^2 \beta^2 \tau^3 \sum_{j = k - \tau}^{k-1} \Big(6\alpha^2 n L^2 \E_k [\|\nabla f(\xbar^{(j-1)})\|^2] + 9L^2 \E_k [\|\hat \bsx^{(j-1)}\|^2] + 3L^2 \E_k [\|\hat \bsx^{(j)}\|^2] \Big) \nonumber \\
        &\phantom{=} \mathrel{+} \Big(6\alpha^4 \beta^2 \tau^4 L^2 + 2\alpha^2 \beta^2 n + \frac{8\alpha^2 \beta^2 \tau^2 n}{1-\beta^2} \Big) \sigma^2. \label{eq:lem-consensus-prf-2}
    \end{align}
    We then sum up~\eqref{eq:lem-consensus-prf-2} over iteration $k$ from $\tau+1$ to $K$ ($K \geq \tau+1$) and obtain
    \begin{align*}
        \MoveEqLeft[0.2] \sum_{k=\tau+1}^K \E_k \big[\big\lVert \hat{\bsx}^{(k)} \big\rVert^2 \big] \\ 
        &\leq 12\alpha^4 \beta^2 \tau^3 n L^2 \sum_{k=\tau+1}^K \sum_{j = k - \tau}^{k-1} \E_k [\|\nabla f(\xbar^{(j-1)})\|^2] + 18\alpha^2 \beta^2 \tau^3 L^2 \sum_{k=\tau+1}^K \sum_{j = k - \tau}^{k-1} \E_k [\|\hat \bsx^{(j-1)}\|^2] \\ 
        &\phantom{\leq} \mathrel{+} 6\alpha^2 \beta^2 \tau^3 L^2 \sum_{k=\tau+1}^K \sum_{j = k - \tau}^{k-1} \E_k [\|\hat \bsx^{(j)}\|^2] + \Big(6\alpha^4 \beta^2 \tau^4 L^2 + 2\alpha^2 \beta^2 n + \frac{8\alpha^2 \beta^2 \tau^2 n}{1-\beta^2} \Big) (K-\tau) \sigma^2.
    \end{align*}
    Adding $\sum_{k=0}^\tau \E_k [\|\hat{\bsx}^{(k)}\|^2]$ and dividing $(K+1)$ on both sides gives
    \begin{align}
        \MoveEqLeft[0.2] \frac{1}{K+1} \sum_{k=0}^K \E_k \big[\big\lVert \hat{\bsx}^{(k)} \big\rVert^2 \big] \nonumber \\
        &\leq \frac{1}{K+1} \sum_{k=0}^\tau \E_k [\|\hat{\bsx}^{(k)}\|^2] + \frac{12\alpha^4 \beta^2 \tau^3 n L^2}{K+1} \sum_{k=\tau+1}^K \sum_{j = k - \tau}^{k-1} \E_k [\|\nabla f(\xbar^{(j-1)})\|^2] \nonumber \\
        &\phantom{\leq} \mathrel{+} \frac{18\alpha^2 \beta^2 \tau^3 L^2}{K+1} \sum_{k=\tau+1}^K \sum_{j = k - \tau}^{k-1} \E_k [\|\hat \bsx^{(j-1)}\|^2] + \frac{6\alpha^2 \beta^2 \tau^3 L^2}{K+1} \sum_{k=\tau+1}^K \sum_{j = k - \tau}^{k-1} \E_k [\|\hat \bsx^{(j)}\|^2] \nonumber \\
        &\phantom{\leq} \mathrel{+} \Big(6\alpha^4 \beta^2 \tau^4 L^2 + 2\alpha^2 \beta^2 n + \frac{8\alpha^2 \beta^2 \tau^2 n}{1-\beta^2} \Big) \frac{K-\tau}{K+1} \sigma^2. \label{eq:lem-consensus-prf-3}
    \end{align}
    Note that for any constant $K \in \natint$ and for any sequence $\{\psi_j\} \subset \reals$, there exists a nonnegative sequence $\{\gamma_j\} \subset \reals_{\geq 0}$ such that $\gamma_j \leq 2\tau$ for all $j \in [K]$ and 
    \[
        \sum_{k=\tau+1}^K \sum^{k-1}_{j=k-\tau} \psi_j = \sum_{k=1}^K \gamma_k \psi_k \leq 2\tau \sum_{k=0}^K \psi_k, \qquad \sum_{k=\tau+1}^K \sum_{j=k-\tau-1}^{k-2} \psi_j = \sum_{k=0}^{K-1} \gamma_k \psi_k \leq 2\tau \sum_{k=0}^K \psi_k.
    \]
    This observation helps simplify the double summation in~\eqref{eq:lem-consensus-prf-3}:
    \begin{align*}
        \MoveEqLeft[0.2] \frac{1}{K+1} \sum_{k=0}^K \E_k \big[\big\lVert \hat{\bsx}^{(k)} \big\rVert^2 \big] \\ 
        &\leq \frac{1}{K+1} \sum_{k=0}^\tau \E_k [\|\hat{\bsx}^{(k)}\|^2] + \frac{24\alpha^4 \beta^2 \tau^4 n L^2}{K+1} \sum_{k=0}^K \E_k [\|\nabla f(\xbar^{(k)})\|^2] \nonumber \\
        &\phantom{\leq} \mathrel{+} \frac{48\alpha^2 \beta^2 \tau^4 L^2}{K+1} \sum_{k=0}^K \E_k [\|\hat \bsx^{(k)}\|^2] + \Big(6\alpha^4 \beta^2 \tau^4 L^2 + 2\alpha^2 \beta^2 n + \frac{8\alpha^2 \beta^2 \tau^2 n}{1-\beta^2} \Big) \frac{K-\tau}{K+1} \sigma^2,
    \end{align*}
    or equivalently,
    \begin{align*}
        \frac{1 - 48\alpha^2 \beta^2 \tau^4 L^2}{K+1} \sum_{k=0}^K \E_k [\|\hat \bsx^{(k)}\|^2] 
        &\leq \frac{1}{K+1} \sum_{k=0}^\tau \E_k [\|\hat{\bsx}^{(k)}\|^2] + \frac{24\alpha^4 \beta^2 \tau^4 n L^2}{K+1} \sum_{k=0}^K \E_k [\|\nabla f(\xbar^{(k)})\|^2] \nonumber \\
        &\phantom{\leq} \mathrel{+} \Big(6\alpha^4 \beta^2 \tau^4 L^2 + 2\alpha^2 \beta^2 n + \frac{8\alpha^2 \beta^2 \tau^2 n}{1-\beta^2} \Big) \frac{K-\tau}{K+1} \sigma^2.
    \end{align*}
    The stepsize condition $\alpha \leq \tfrac{1}{4\sqrt{6} \beta \tau^2 L}$ implies that $1 - 48\alpha^2 \beta^2 \tau^4 L^2 \geq \tfrac{1}{2}$, and thus 
    \begin{align*}
        \frac{1}{K+1} \sum_{k=0}^K \E_k [\|\hat \bsx^{(k)}\|^2] &\leq \frac{2}{K+1} \sum_{k=0}^\tau \E_k [\|\hat{\bsx}^{(k)}\|^2] + \frac{48\alpha^4 \beta^2 \tau^4 n L^2}{K+1} \sum_{k=0}^K \E_k [\|\nabla f(\xbar^{(k)})\|^2] \\ 
        &\phantom{\leq} \mathrel{+} \Big(12\alpha^4 \beta^2 \tau^4 L^2 + 4\alpha^2 \beta^2 n + \frac{16\alpha^2 \beta^2 \tau^2 n}{1-\beta^2} \Big) \frac{K-\tau}{K+1} \sigma^2 \\ 
        &\leq \frac{2}{K+1} \sum_{k=0}^\tau \E_k [\|\hat{\bsx}^{(k)}\|^2] + \frac{n}{192\beta^2 \tau^4 L^2 (K+1)} \sum_{k=0}^K \E_k [\|\nabla f(\xbar^{(k)})\|^2] \\ 
        &\phantom{\leq} \mathrel{+} \Big(\frac{1}{768 \beta^2 \tau^4 L^2} + \frac{n}{24 \tau^4 L^2} + \frac{n}{6 (1-\beta^2) \tau^2 L^2}\Big) \sigma^2,
    \end{align*}
    where the second inequality uses the stepsize condition and relax $\tfrac{K-\tau}{K+1}$ to $1$. Taking the total expectation yields the desired conclusion.
\end{proof}

With \cref{lem:descent,lem:consensus}, we are ready to present the proof of the main result of this paper.
\begin{proof}[Proof of \cref{thm:conv}]
    Combining \cref{lem:descent} with \eqref{eq:lem-consensus} yields
    \begin{align*}
        \MoveEqLeft[0.2] \frac{1}{K+1} \sum_{k=0}^K \Big( \E \|\overline{\nabla f} (\bsx^{(k)})\|^2 + \E \|\nabla f(\xbar^{(k)})\|^2 \Big) \\ 
        &\leq \frac{4L^2}{n(K+1)} \sum_{k=0}^\tau \E [\|\hat \bsx^{(k)}\|^2] + \frac{96 \alpha^4 \beta^2 \tau^4 n L^4}{K+1} \sum_{k=0}^K \E [\|\nabla f(\xbar^{(k)})\|^2] \\
        &\phantom{\leq} \mathrel{+} \frac{4}{\alpha (K+1)} \E \ftilde (\xbar^{(0)}) + \Big(\frac{24\alpha^4 \beta^2 \tau^4 L^4 + 2\alpha L}{n} + 8\alpha^2 \beta^2 L^2 + \frac{32\alpha^2 \beta^2 \tau^2 L^2}{1-\beta^2} \Big) \sigma^2 \\ 
        &\leq \frac{4L^2}{n(K+1)} \sum_{k=0}^\tau \E [\|\hat \bsx^{(k)}\|^2] + \frac{96 \alpha^4 \beta^2 \tau^4 n L^4}{K+1} \sum_{k=0}^K \big(\E \|\overline{\nabla f} (\bsx^{(k)})\|^2 + \E [\|\nabla f(\xbar^{(k)})\|^2] \big) \\
        &\phantom{\leq} \mathrel{+} \frac{4}{\alpha (K+1)} \E \ftilde (\xbar^{(0)}) + \Big(\frac{24\alpha^4 \beta^2 \tau^4 L^4 + 2\alpha L}{n} + 8\alpha^2 \beta^2 L^2 + \frac{32\alpha^2 \beta^2 \tau^2 L^2}{1-\beta^2} \Big) \sigma^2.
    \end{align*}
    Grouping similar terms on the left-hand side gives
    \begin{align*}
        \MoveEqLeft[0.2] \frac{1 - 96 \alpha^4 \beta^2 \tau^4 n L^4}{K+1} \sum_{k=0}^K \big(\E \|\overline{\nabla f} (\bsx^{(k)})\|^2 + \E [\|\nabla f(\xbar^{(k)})\|^2] \big) \\ 
        &\leq \frac{4}{\alpha (K+1)} \E \ftilde (\xbar^{(0)}) + \frac{4L^2}{n(K+1)} \sum_{k=0}^\tau \E [\|\hat \bsx^{(k)}\|^2] \\ 
        &\phantom{\leq} \mathrel{+} \Big(\frac{24\alpha^4 \beta^2 \tau^4 L^4 + 2\alpha L}{n} + 8\alpha^2 \beta^2 L^2 + \frac{32\alpha^2 \beta^2 \tau^2 L^2}{1-\beta^2} \Big) \sigma^2.
    \end{align*}
    The stepsize condition $\alpha \leq \frac{1}{4\sqrt{6} \beta \tau^2 L}$ implies that
    \begin{align}
        \MoveEqLeft[0.2] \frac{1}{K+1} \sum_{k=0}^K \big(\E \|\overline{\nabla f} (\bsx^{(k)})\|^2 + \E [\|\nabla f(\xbar^{(k)})\|^2] \big) \nonumber \\ 
        &\leq \frac{8}{\alpha (K+1)} \E \ftilde (\xbar^{(0)}) + \frac{8L^2}{n(K+1)} \sum_{k=0}^\tau \E [\|\hat \bsx^{(k)}\|^2] \nonumber \\ 
        &\phantom{\leq} \mathrel{+} \Big(\frac{48\alpha^4 \beta^2 \tau^4 L^4 + 4\alpha L}{n} + 16\alpha^2 \beta^2 L^2 + \frac{64\alpha^2 \beta^2 \tau^2 L^2}{1-\beta^2} \Big) \sigma^2 \nonumber \\ 
        &\leq \frac{32\sqrt{6} \beta \tau^2 L}{K+1} \E \ftilde (\xbar^{(0)}) + \frac{8L^2}{n(K+1)} \sum_{k=0}^\tau \E [\|\hat \bsx^{(k)}\|^2] + \frac{\sigma^2}{192 \beta^2 \tau^4 n} + \frac{\sigma^2}{\sqrt{6} \beta \tau^2 n} + \frac{\sigma^2}{6\tau^4} + \frac{2 \sigma^2}{3 (1-\beta^2) \tau^2}, \label{eq:thm-prf}
    \end{align}
    where in the last step we plug in the stepsize condition.
\end{proof}

Now we state the proof of \cref{cor:conv}. To improve the readability of the proof, from now on, we use the notation $\lesssim$ to hide irrelevant constants. The notation $a \lesssim b$ means that there exists a positive constant $\gamma \in \reals_{>0}$ such that $a \leq \gamma b$. In our case, the important quantities that we keep are $\alpha$, $\beta$, $\tau$, $n$, $L$, and $\sigma$.
\begin{proof}[Proof of \cref{cor:conv}]
    From the stepsize condition $\alpha \leq \min \big\{\tfrac{1}{2L}, \tfrac{1}{4\sqrt{6} \beta \tau^2 L}\big\}$, the inequality~\eqref{eq:thm-prf} becomes
    \begin{align}
        \MoveEqLeft[0.2] \frac{1}{K+1} \sum\limits_{k=0}^K \big(\E \|\overline{\nabla f} (\bsx^{(k)})\|^2 + \E [\|\nabla f(\xbar^{(k)})\|^2] \big) \nonumber \\ 
        &\lesssim \frac{L^2}{\alpha K} + \frac{\alpha L \sigma^2}{n} + \frac{\alpha^4 \beta^2 \tau^4 L^4 \sigma^2}{n} + \frac{\alpha^2 \beta^2 \tau^2 L^2 \sigma^2}{1-\beta^2} \nonumber \\ 
        &\lesssim \frac{L^2}{\alpha K} + \frac{\alpha L \sigma^2}{n} + \frac{\alpha^2 \alpha^2 \beta^2 \tau^4 L^4 \sigma^2}{n} + \frac{\alpha^2 \beta^2 \tau^2 L^2 \sigma^2}{1-\beta^2} \nonumber \\ 
        &\lesssim \frac{L^2}{\alpha K} + \frac{\alpha L \sigma^2}{n} + \frac{\alpha^2 L^2 \sigma^2}{n} + \frac{\alpha^2 \beta^2 \tau^2 L^2 \sigma^2}{1-\beta^2} \nonumber \\ 
        &\lesssim \frac{L^2}{\alpha K} + \frac{\alpha L \sigma^2}{n} + \frac{\alpha^2 \beta^2 \tau^2 L^2 \sigma^2}{1-\beta^2}, \label{eq:cor-prf-1} \\ 
        &\lesssim \frac{c_1}{\alpha K} + c_2 \alpha + c_3 \alpha^2, \nonumber
    \end{align}
    where $c_1 = L^2$, $c_2 = \tfrac{L\sigma^2}{n}$, and $c_3 = \frac{\beta^2 \tau^2 L^2 \sigma^2}{1-\beta^2}$. In~\eqref{eq:cor-prf-1} we use the assumption that $\tfrac{1}{n} \ll \beta^2 \tau^2$. Now we set the stepsize $\alpha$ as in \eqref{eq:cor-alpha}. (Note that by definition, this choice of $\alpha$ satisfies the condition in \cref{thm:conv}.) We then discuss the following three cases.
    \begin{enumerate}
        \item If $\alpha = \min \left\{\frac{1}{2L}, \frac{1}{4\sqrt{6} \beta \tau^2 L}\right\} \leq \min\left\{ \left(\tfrac{c_1}{c_2 K} \right)^{\frac{1}{2}}, \left(\frac{c_1}{c_3 K} \right)^{\frac{1}{2}} \right\}$, then
        \[
            \frac{c_1}{\alpha K} + c_2 \alpha + c_3 \alpha^2 \lesssim \frac{c_1}{\alpha K} + \left(\frac{c_1 c_2}{K} \right)^{\frac{1}{2}} + \left(\frac{c_1 c_3}{K} \right)^{\frac{1}{2}}.
        \]

        \item If $\alpha = \left(\frac{c_1}{c_2 K} \right)^{\frac{1}{2}} \leq \left(\frac{c_1}{c_3 K} \right)^{\frac{1}{2}}$, then
        \[
            \frac{c_1}{\alpha K} + c_2 \alpha + c_3 \alpha^2 \lesssim \left(\frac{c_1 c_2}{K} \right)^{\frac{1}{2}} + \left(\frac{c_1 c_3}{K} \right)^{\frac{1}{2}}.
        \]

        \item If $\alpha = \left(\frac{c_1}{c_3 K} \right)^{\frac{1}{2}} \leq \left(\frac{c_1}{c_2 K} \right)^{\frac{1}{2}}$, then
        \[
            \frac{c_1}{\alpha K} + c_2 \alpha + c_3 \alpha^2 \lesssim \left(\frac{c_1 c_2}{K} \right)^{\frac{1}{2}} + \left(\frac{c_1 c_3}{K} \right)^{\frac{1}{2}}.
        \]
    \end{enumerate}
    Combing all three cases yields
    \[
        \frac{c_1}{\alpha K} + c_2 \alpha + c_3 \alpha^2 \lesssim \frac{c_1}{\alpha K} + \left(\frac{c_1 c_2}{K} \right)^{\frac{1}{2}} + \left(\frac{c_1 c_3}{K} \right)^{\frac{1}{2}} \lesssim \frac{\beta \tau^2 L^3}{K} + \frac{L^{\frac{3}{2}} \sigma}{(nK)^{\frac{1}{2}}} + \frac{\beta \tau L^2 \sigma}{(1-\beta^2)^{\frac{1}{2}} K^{\frac{1}{2}}}. \qedhere
    \]
\end{proof}

\section{Numerical experiments} \label{sec:exp}

This numerical experiments presented in this section illustrate how GT-PGA accelerates practical convergence compared to vanilla GT and LU-GT. We apply GT-PGA to solve the least squares problem with a nonconvex regularization term:
\begin{equation} \label{eq:exp-prob}
    \mini \quad \frac{1}{n} \sum_{i=1}^n \|A_i x - b_i\|_2^2 + \lambda \sum_{j=1}^d \frac{x[j]}{1+x[j]},
\end{equation}
where the decision variable is $x \in \reals^d$, $x[j]$ is the $j$th component of~$x$, and $\{(A_i,b_i)\}_{i=1}^n \subset \reals^{m_i \times d} \times \reals^{m_i}$ are the local data held by agent~$i$. In our simulation, we set $\lambda=0.01$, $n=64$, $d=20$, $m_i = 500$ for all $i \in [n]$, and the entries of each $A_i$ are drawn independently from standard Gaussian distribution. For all $i \in [n]$, we randomly generate $\xtilde_i \in \reals^d$ and set $b_i = A_i \xtilde_i + z_i$, where $z_i \sim \cN(0, 0.01)$ are drawn independently. We test \cref{alg:gt} on various topologies, including the ring graph, 2D-MeshGrid, star graph, and static hypercuboid \cite{NJYU23}.
\begin{figure}[!htb]
\centering
\begin{subfigure}{.5\columnwidth}
    \centering
    \includegraphics[width=\linewidth]{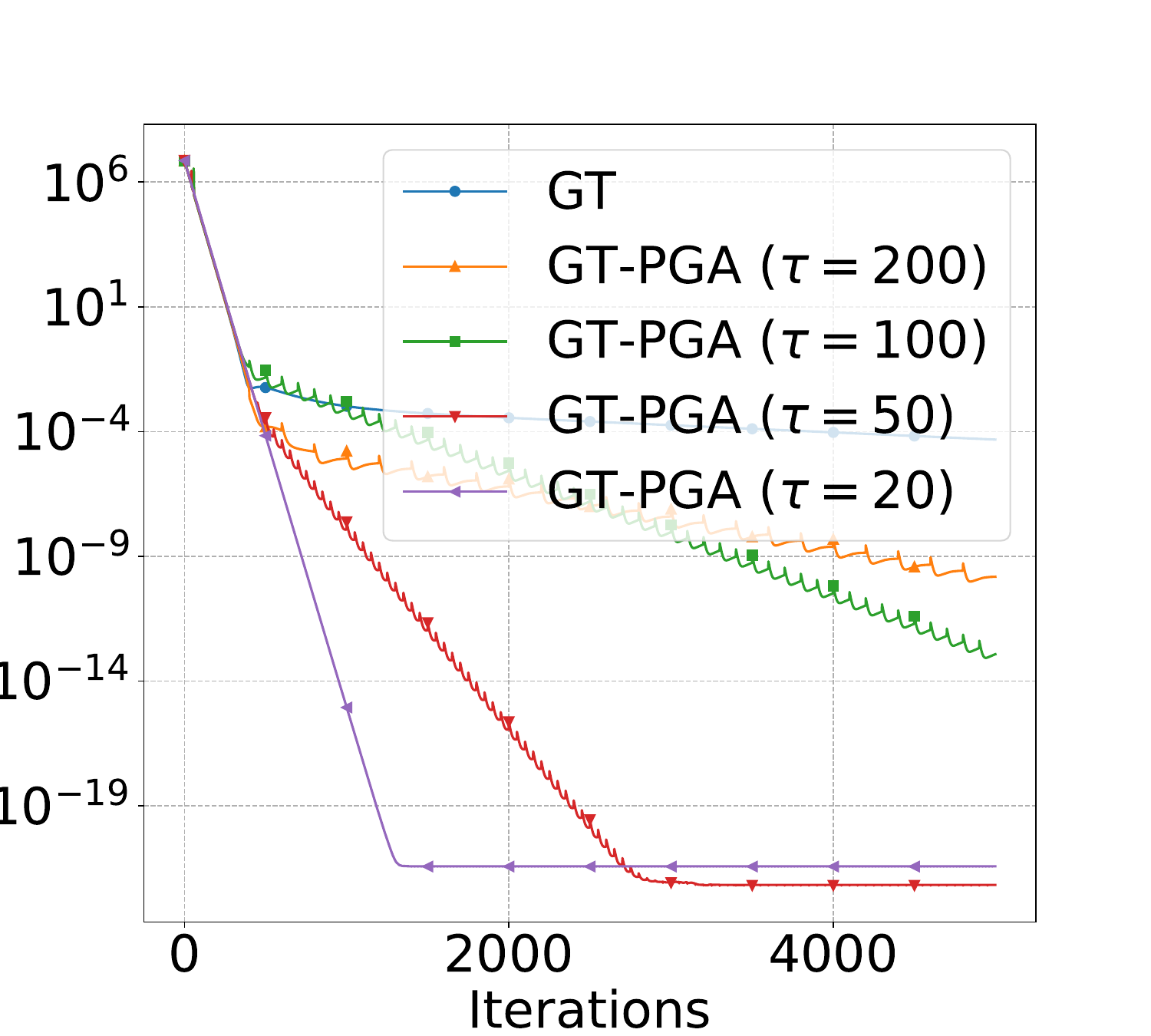}
    \caption{Ring.}
    \label{fig:Ring}
\end{subfigure}%
\begin{subfigure}{.5\columnwidth}
    \centering
    \includegraphics[width=\linewidth]{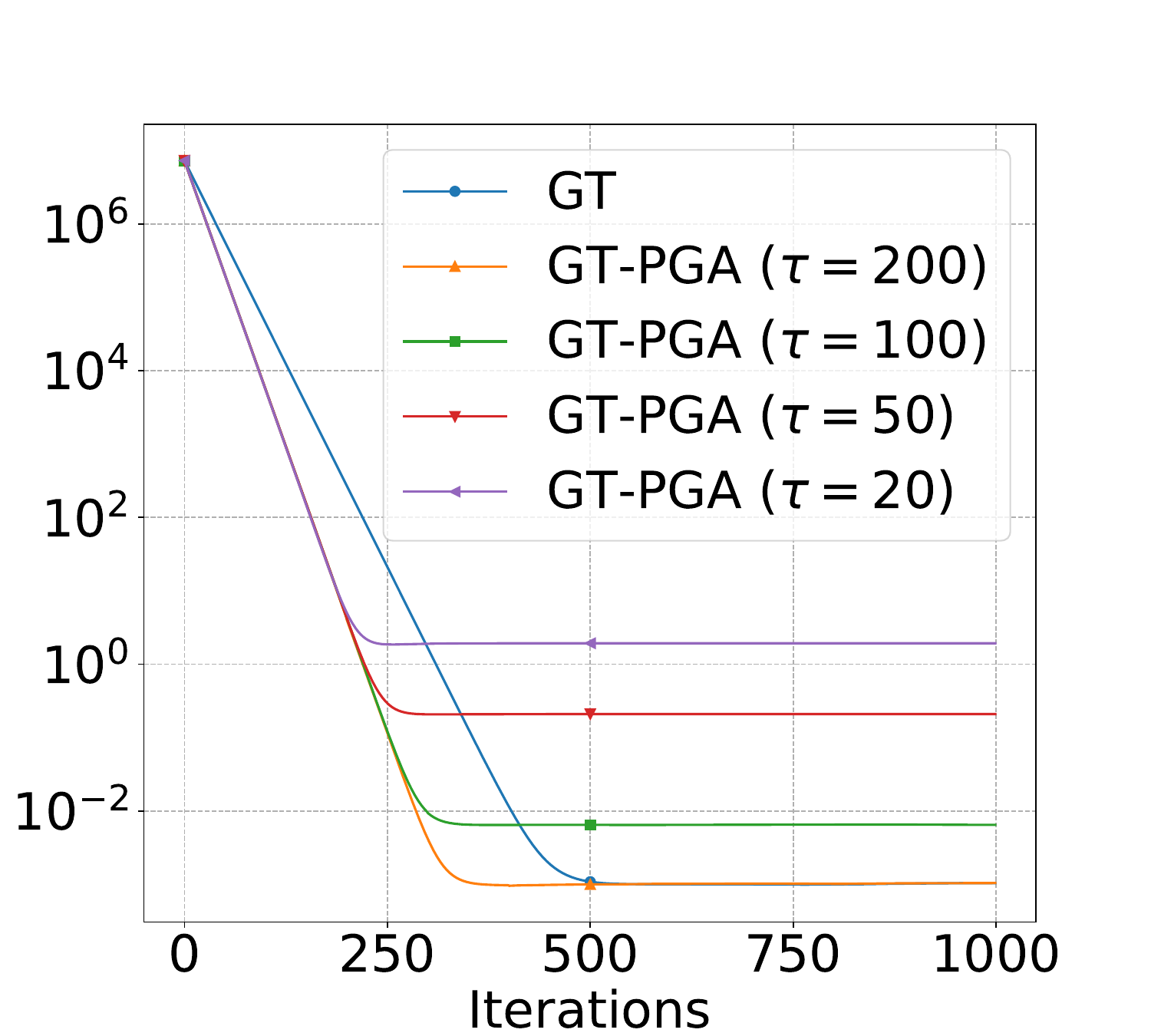}
    \caption{2D-MeshGrid.}
    \label{fig:2D-MeshGrid}
\end{subfigure}
\begin{subfigure}{.5\columnwidth}
    \centering
    \includegraphics[width=\linewidth]{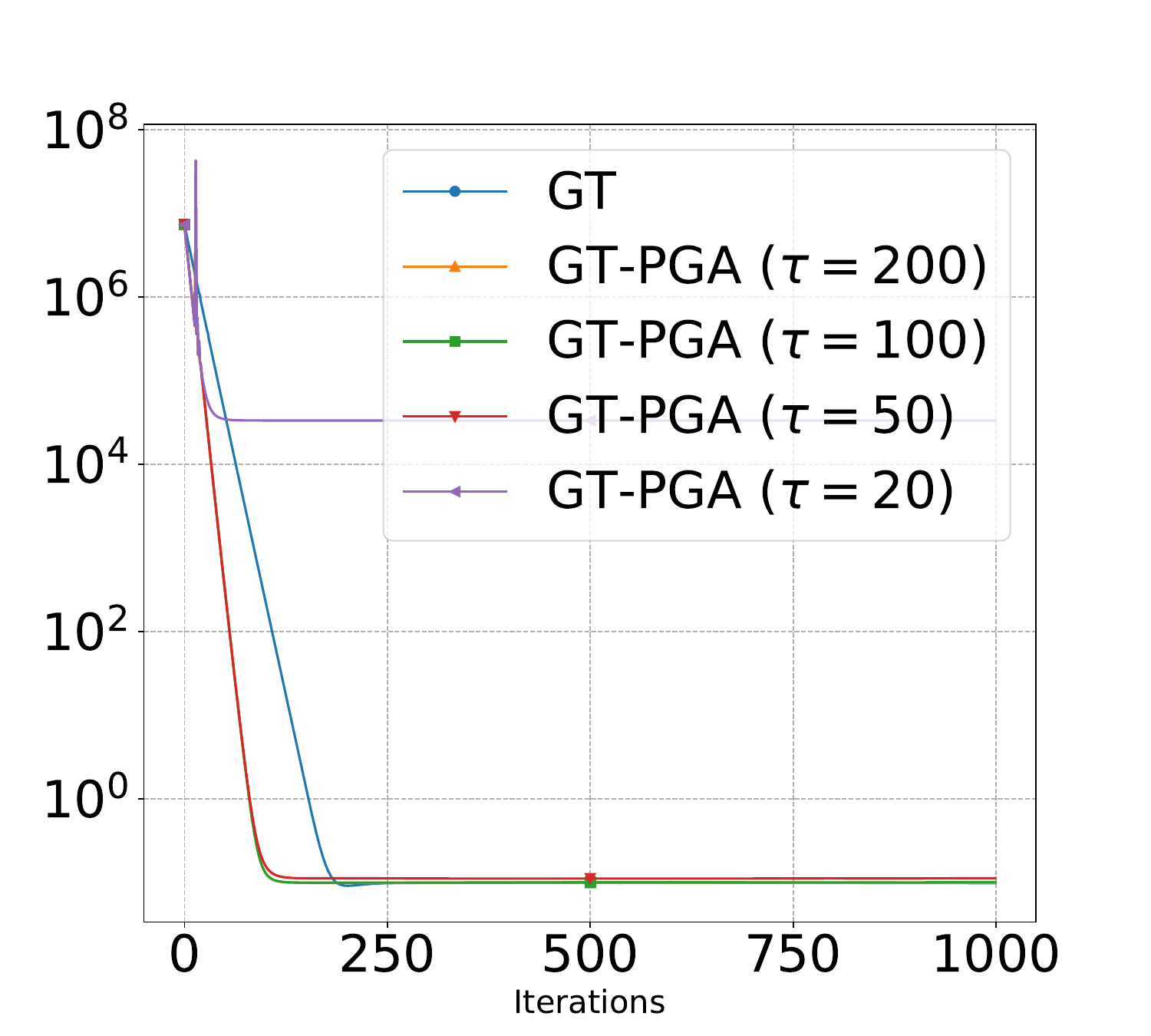}
    \caption{Star.}
    \label{fig:Star-Graph}
\end{subfigure}%
\begin{subfigure}{.5\columnwidth}
    \centering
    \includegraphics[width=\linewidth]{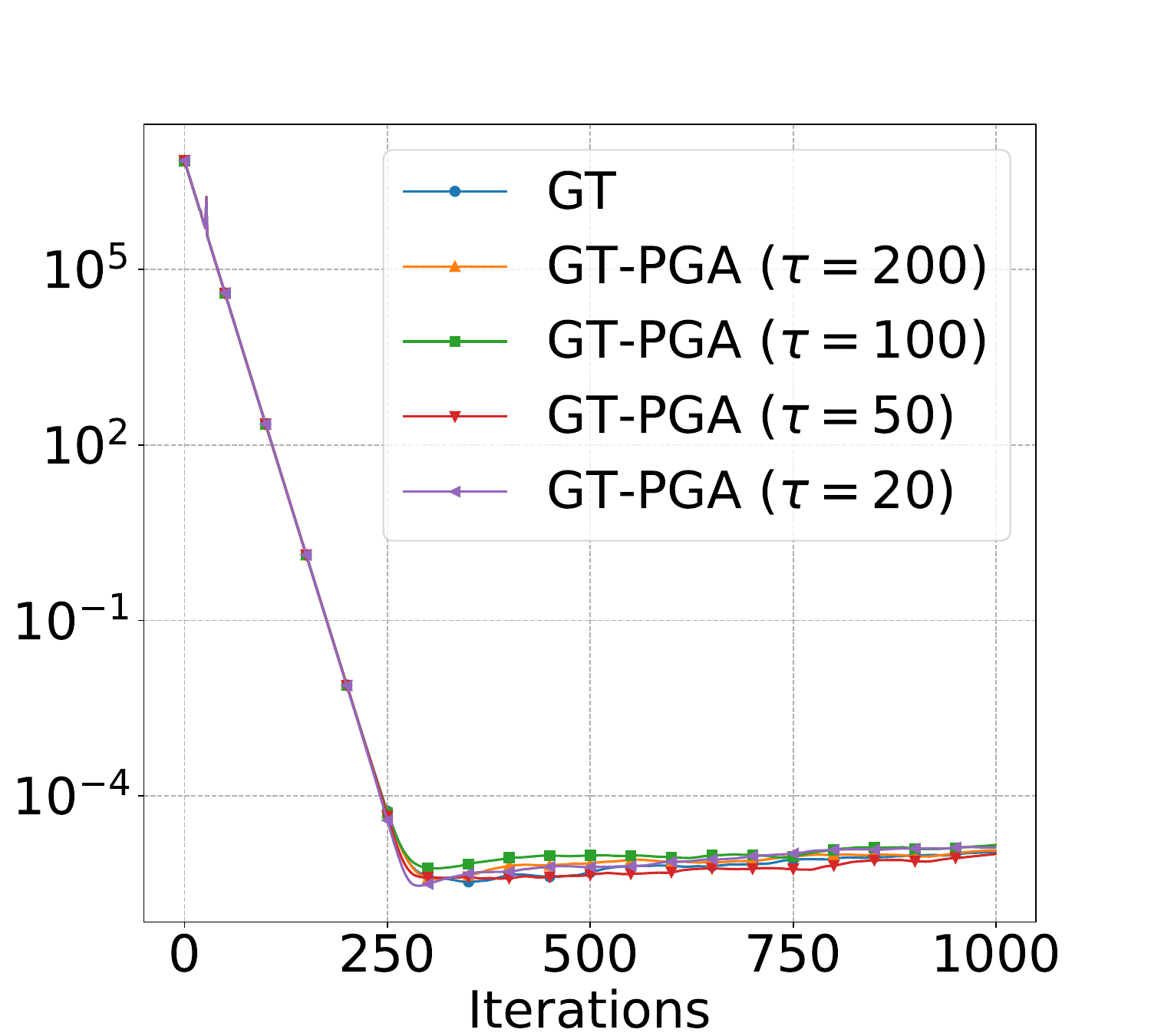}
    \caption{Static hypercuboid.}
    \label{fig:Static-Hypercuboid}
\end{subfigure}
\caption{Performance of GT-PGA for solving~\eqref{eq:exp-prob} with various topologies. The plots report $ \|\overline{\nabla f} (\bsx^{(k)})\|^2 +  \|\nabla f(\xbar^{(k)})\|^2$, and use the ring, 2D-MeshGrid, the star graph, and the static hypercuboid, respectively. Different curves on each figure use different PGA period, \ie, $\tau=20, 50, 100, 200$, and $\infty$ (equivalent to vanilla GT).}
\label{fig:result}
\end{figure}

The simulation results depicted in \cref{fig:result} show the superiority of the PGA operation and align with the theoretical insight from \cref{thm:conv} and \cref{cor:conv}. For the sparse ring graph (\cref{fig:Ring}), PGA helps reduce the stochastic noise caused by stochastic gradients, and GT-PGA converges to a more accurate solution compared to vanilla GT. For 2D-MeshGrid and the star graph, GT-PGA exhibits a better practical performance in terms of convergence rate. For static hyper-cuboids, the benefit of GT-PGA is marginal, potentially because of the desirable properties of static hyper-cuboids \cite{NJYU23}. The simulation results in \cref{fig:result} also reveal the inherent properties of graphs: sparser graphs require more frequent global averaging operations. To provide a concrete example, each node in the ring graph has fewer neighbors compared to 2D MeshGrid, so information is harder to be evenly distributed across nodes through communication. As a result, a higher frequency of global averaging operations tends to yield better performance for the ring graph.

The simulation results shown in \cref{fig:result-LU-GT} demonstrate that the GT-PGA algorithm significantly outperforms LU-GT across all topologies, which is consistent with the theoretical insights from \cref{thm:conv}.

\begin{figure}[!htb]
\centering
\begin{subfigure}{.5\columnwidth}
    \centering
    \includegraphics[width=\linewidth]{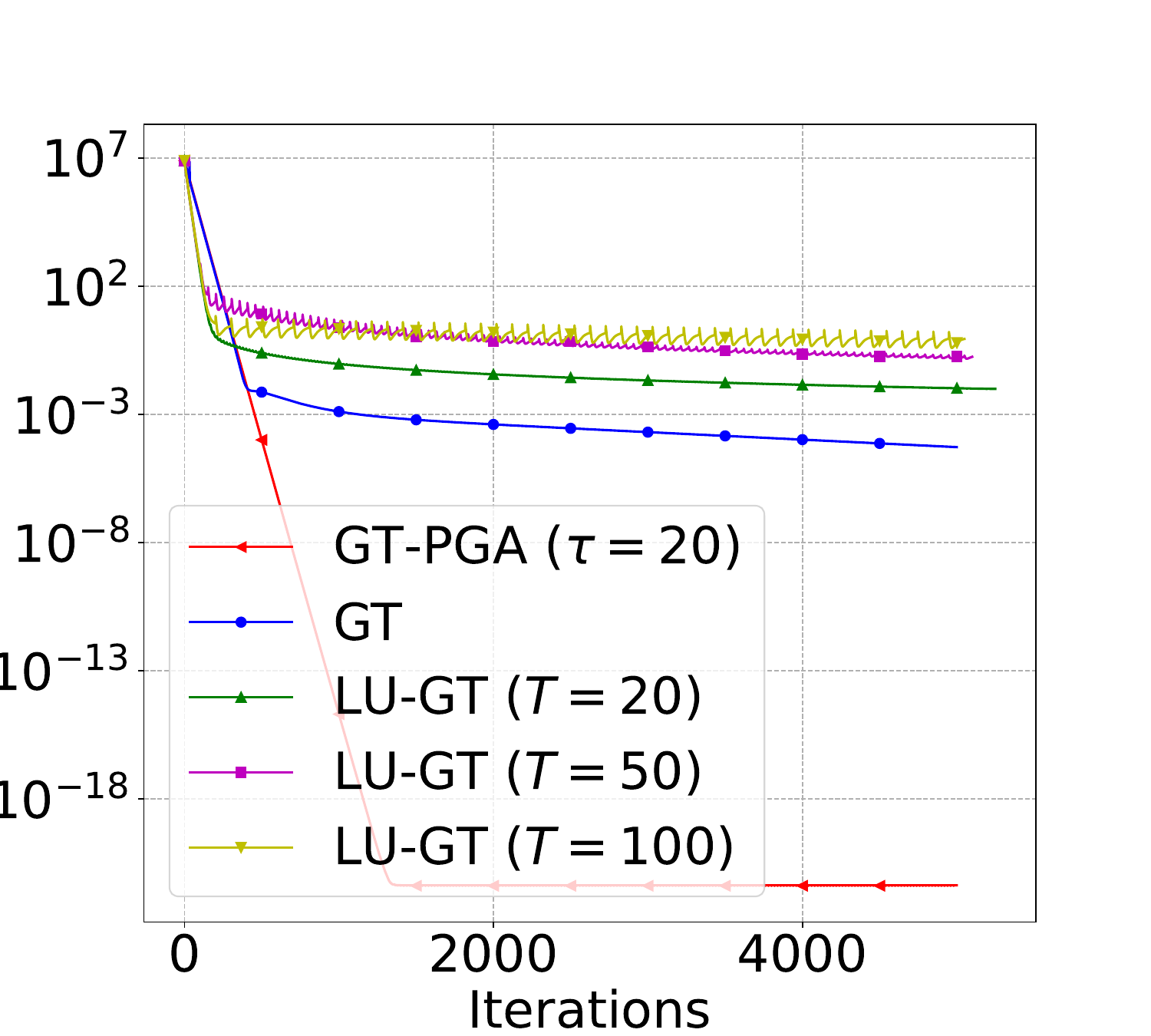}
    \caption{Ring.}
    \label{fig:Ring-LU-GT}
\end{subfigure}%
\begin{subfigure}{.5\columnwidth}
    \centering
    \includegraphics[width=\linewidth]{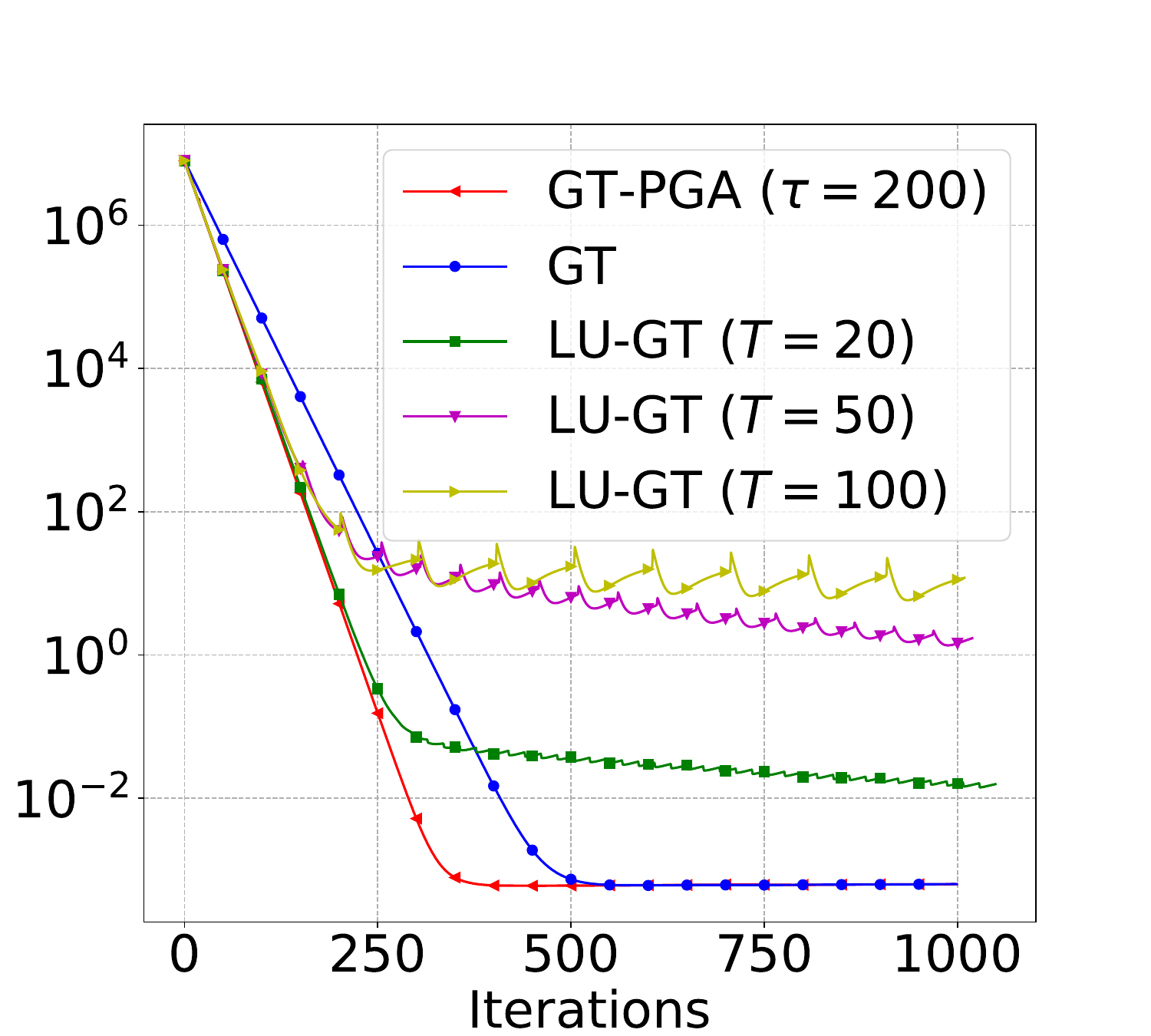}
    \caption{2D-MeshGrid.}
    \label{fig:2D-MeshGrid-LU-GT}
\end{subfigure}
\begin{subfigure}{.5\columnwidth}
    \centering
    \includegraphics[width=\linewidth]{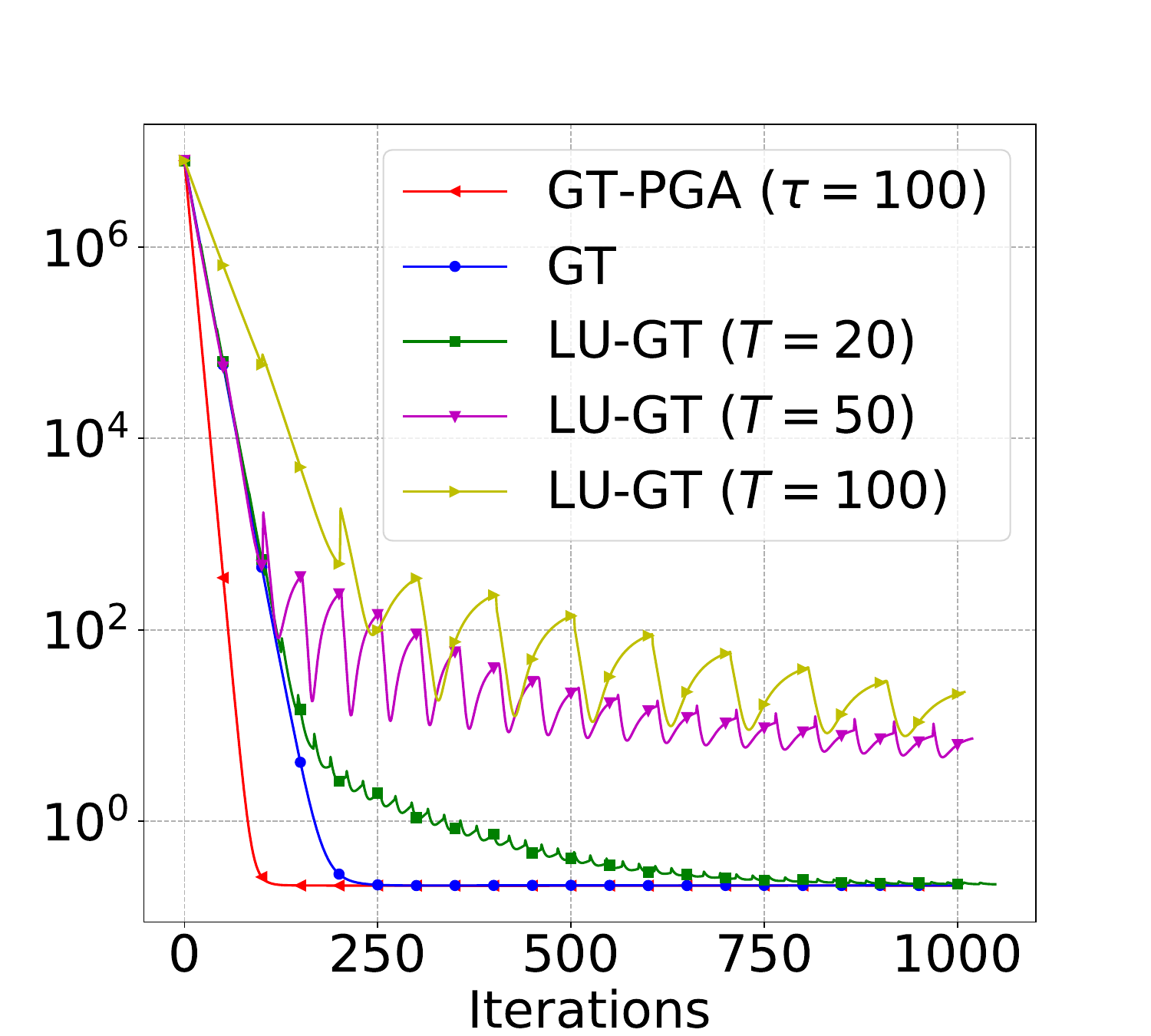}
    \caption{Star.}
    \label{fig:Star-Graph-LU-GT}
\end{subfigure}%
\begin{subfigure}{.5\columnwidth}
    \centering
    \includegraphics[width=\linewidth]{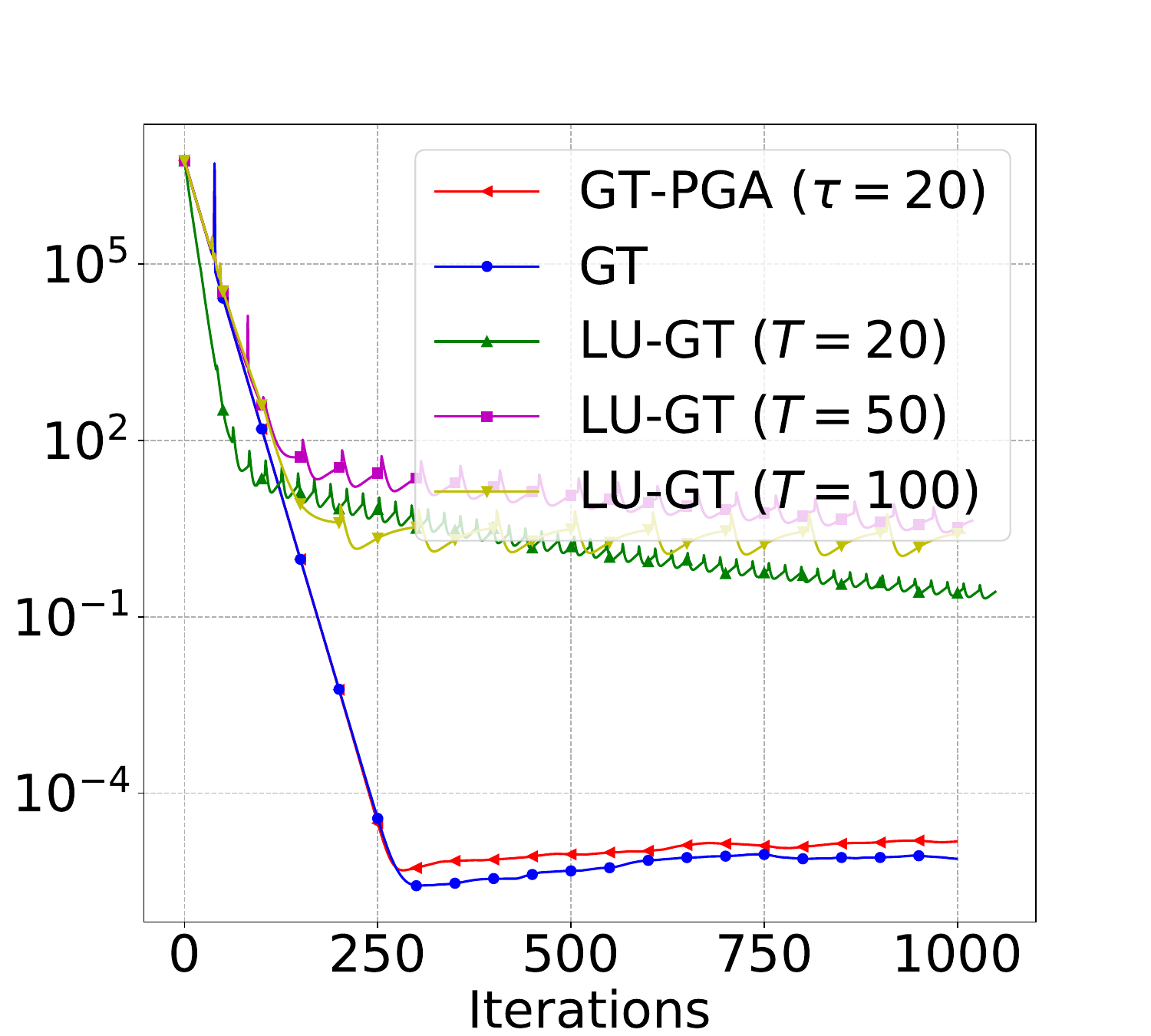}
    \caption{Static hypercuboid.}
    \label{fig:Static-Hypercuboid-LU-GT}
\end{subfigure}
\caption{Comparison of GT-PGA and LU-GT across various topologies, including the ring, 2D MeshGrid, star graph, and static hypercuboid. The plots also drew $ \|\overline{\nabla f} (\bsx^{(k)})\|^2 +  \|\nabla f(\xbar^{(k)})\|^2$. Different curves on each figure represent that the algorithm performs T computations in Local Updates step, \ie, $T=1$ (equivalent to vanilla GT), $20, 50, 100$\cite{BBG23a}. }
\label{fig:result-LU-GT}
\end{figure}

\section{Conclusion} \label{sec:conclusion}

We incorporate periodic global averaging (PGA) into Gradient Tracking (GT) and propose a new decentralized algorithm GT-PGA. We establish convergence guarantees for GT-PGA under the stochastic, nonconvex setting and showcase the superiority of GT-PGA compared with vanilla GT. Numerical results validate the improvements in practical convergence due to the proposed periodic global averaging operation. While we focus on the nonconvex setting (due to space constraints), it is straightforward to extend our analysis to the convex setting.

In this work, we focus on a specific form of GT \cite{DS16}. It is not clear whether PGA can be incorporated into other forms of GT and whether a unified analysis similar to, \eg, \cite{AY22}, still holds. Moreover, the connection (or difference) between PGA and multi-consensus is still unknown, and further analysis is needed to quantify the trade-off between these two techniques.

\bibliography{reference}

\begin{thebibliography}{10}

\bibitem{Alghunaim24}
Sulaiman~A. Alghunaim.
\newblock Local exact-diffusion for decentralized optimization and learning.
\newblock {\em IEEE Transactions on Automatic Control}, 2024.

\bibitem{AY22}
Sulaiman~A. Alghunaim and Kun Yuan.
\newblock A unified and refined convergence analysis for non-convex decentralized learning.
\newblock {\em IEEE Transactions on Signal Processing}, 70:3264--3279, 2022.

\bibitem{BBG23a}
Albert~S. Berahas, Raghu Bollapragada, and Shagun Gupta.
\newblock Balancing communication and computation in gradient tracking algorithms for decentralized optimization.
\newblock {\em arXiv e-prints arXiv:2303.14289}, 2023.

\bibitem{BPC+11}
Stephen~P. Boyd, Neal Parikh, Eric Chu, Borja Peleato, and Jonathan Eckstein.
\newblock Distributed optimization and statistical learning via alternating direction method of multipliers.
\newblock {\em Foundations and Trends in Machine Learning}, 3(1):1--122, 2011.

\bibitem{CYRC12}
Yongcan Cao, Wenwu Yu, Wei Ren, and Guanrong Chen.
\newblock An overview of recent progress in the study of distributed multi-agent coordination.
\newblock {\em IEEE Transactions on Industrial informatics}, 9(1):427--438, 2012.

\bibitem{CS10}
Federico~S. Cattivelli and Ali~H. Sayed.
\newblock Diffusion {LMS} strategies for distributed estimation.
\newblock {\em IEEE Transactions on Signal Processing}, 58(3):1035--1048, 2010.

\bibitem{CS13}
Jianshu Chen and Ali~H. Sayed.
\newblock Distributed {Pareto} optimization via diffusion strategies.
\newblock {\em IEEE Journal of Selected Topics in Signal Processing}, 7(2):205--220, 2013.

\bibitem{CYZ+21}
Yiming Chen, Kun Yuan, Yingya Zhang, Pan Pan, Yinghui Xu, and Wotao Yin.
\newblock Accelerating gossip {SGD} with periodic global averaging.
\newblock In {\em International Conference on Machine Learning}, pages 1791--1802, 2021.

\bibitem{DS16}
Paolo Di~Lorenzo and Gesualdo Scutari.
\newblock Next: In-network nonconvex optimization.
\newblock {\em IEEE Transactions on Signal and Information Processing over Networks}, 2(2):120--136, 2016.

\bibitem{GC23}
Songyang Ge and Tsung-Hui Chang.
\newblock Gradient tracking with multiple local {SGD} for decentralized non-convex learning.
\newblock In {\em 62nd IEEE Conference on Decision and Control (CDC)}, pages 133--138, 2023.

\bibitem{GHR21}
Eduard Gorbunov, Filip Hanzely, and Peter Richt{\'a}rik.
\newblock Local {SGD}: {Unified} theory and new efficient methods.
\newblock In {\em International Conference on Artificial Intelligence and Statistics}, volume 130, pages 3556--3564, 2021.

\bibitem{KKM+20}
Sai~Praneeth Karimireddy, Satyen Kale, Mehryar Mohri, Sashank Reddi, Sebastian~U. Stich, and Ananda~Theertha Suresh.
\newblock {SCAFFOLD}: Stochastic controlled averaging for federated learning.
\newblock In {\em International Conference on Machine Learning}, pages 5132--5143, 2020.

\bibitem{KMR20}
Ahmed Khaled, Konstantin Mishchenko, and Peter Richt{\'a}rik.
\newblock Tighter theory for local {SGD} on identical and heterogeneous data.
\newblock In {\em International Conference on Artificial Intelligence and Statistics}, pages 4519--4529, 2020.

\bibitem{KLB+20}
Anastasia Koloskova, Nicolas Loizou, Sadra Boreiri, Martin Jaggi, and Sebastian~U. Stich.
\newblock A unified theory of decentralized {SGD} with changing topology and local updates.
\newblock In {\em International Conference on Machine Learning}, pages 5381--5393, 2020.

\bibitem{LSY19}
Zhi Li, Wei Shi, and Ming Yan.
\newblock A decentralized proximal-gradient method with network independent step-sizes and separated convergence rates.
\newblock {\em IEEE Transactions on Signal Processing}, 67(17):4494--4506, 2019.

\bibitem{LLKS23}
Yue Liu, Tao Lin, Anastasia Koloskova, and Sebastian~U. Stich.
\newblock Decentralized gradient tracking with local steps.
\newblock {\em arXiv e-prints arXiv:2301.01313}, 2023.

\bibitem{MMSR22}
Konstantin Mishchenko, Grigory Malinovsky, Sebastian~U. Stich, and Peter Richt{\'a}rik.
\newblock {ProxSkip: Yes! Local gradient steps provably lead to communication acceleration! Finally!}
\newblock In {\em International Conference on Machine Learning}, 2022.

\bibitem{MJPH21}
Aritra Mitra, Rayana Jaafar, George~J. Pappas, and Hamed Hassani.
\newblock Linear convergence in federated learning: {Tackling} client heterogeneity and sparse gradients.
\newblock In {\em Advances in Neural Information Processing Systems}, 2021.

\bibitem{NO09}
A.~Nedi{\'c} and A.~Ozdaglar.
\newblock Distributed subgradient methods for multi-agent optimization.
\newblock {\em IEEE Transactions on Automatic Control}, 54(1):48--61, 2009.

\bibitem{NL18}
Angelia Nedi{\'c} and Ji~Liu.
\newblock Distributed optimization for control.
\newblock {\em Annual Review of Control, Robotics, and Autonomous Systems}, 1:77--103, 2018.

\bibitem{NOR18}
Angelia Nedi{\'c}, Alex Olshevsky, and Michael~G Rabbat.
\newblock Network topology and communication-computation tradeoffs in decentralized optimization.
\newblock {\em Proceedings of the IEEE}, 106(5):953--976, 2018.

\bibitem{NOS17}
Angelia Nedi{\'c}, Alex Olshevsky, and Wei Shi.
\newblock Achieving geometric convergence for distributed optimization over time-varying graphs.
\newblock {\em SIAM Journal on Optimization}, 27(4):2597--2633, 2017.

\bibitem{NAYU23}
Edward Duc~Hien Nguyen, Sulaiman~A. Alghunaim, Kun Yuan, and C{\'e}sar~A. Uribe.
\newblock On the performance of gradient tracking with local updates.
\newblock In {\em 62nd {{IEEE Conference}} on {{Decision}} and {{Control}} ({{CDC}})}, pages 4309--4313, 2023.

\bibitem{NJYU23}
Edward Duc~Hien Nguyen, Xin Jiang, Bicheng Ying, and C{\'e}sar~A. Uribe.
\newblock On graphs with finite-time consensus and their use in gradient tracking.
\newblock {\em arXiv e-prints arXiv:2311.01317}, 2023.

\bibitem{PY09}
Pitch Patarasuk and Xin Yuan.
\newblock Bandwidth optimal all-reduce algorithms for clusters of workstations.
\newblock {\em Journal of Parallel and Distributed Computing}, 69(2):117--124, 2009.

\bibitem{PKP06}
Joel~B. Predd, Sanjeev~B. Kulkarni, and Vincent~H. Poor.
\newblock Distributed learning in wireless sensor networks.
\newblock {\em IEEE Signal Processing Magazine}, 23(4):56--69, 2006.

\bibitem{QL18}
Guannan Qu and Na~Li.
\newblock Harnessing smoothness to accelerate distributed optimization.
\newblock {\em IEEE Transactions on Control of Network Systems}, 5(3):1245--1260, 2018.

\bibitem{SLWY15}
Wei Shi, Qing Ling, Gang Wu, and Wotao Yin.
\newblock {EXTRA}: An exact first-order algorithm for decentralized consensus optimization.
\newblock {\em SIAM Journal on Optimization}, 25(2):944--966, 2015.

\bibitem{SLJ+22}
Zhuoqing Song, Weijian Li, Kexin Jin, Lei Shi, Ming Yan, Wotao Yin, and Kun Yuan.
\newblock {Communication-efficient topologies for decentralized learning with {$O(1)$} consensus rate}.
\newblock In {\em Advances in Neural Information Processing Systems}, volume~35, pages 1073--1085, 2022.

\bibitem{Stich18}
Sebastian~U. Stich.
\newblock Local {SGD} converges fast and communicates little.
\newblock In {\em International Conference on Learning Representations}, 2019.

\bibitem{RNV10}
Srinivasan Sundhar~Ram, Angelia Nedi{\'c}, and Venugopal~V. Veeravalli.
\newblock Distributed stochastic subgradient projection algorithms for convex optimization.
\newblock {\em Journal of Optimization Theory and Applications}, 147(3):516--545, 2010.

\bibitem{TLY+18}
Hanlin Tang, Xiangru Lian, Ming Yan, Ce~Zhang, and Ji~Liu.
\newblock {D$^2$}: {D}ecentralized training over decentralized data.
\newblock In {\em International Conference on Machine Learning}, pages 4848--4856, 2018.

\bibitem{XZSX15}
Jinming Xu, Shanying Zhu, Yeng~Chai Soh, and Lihua Xie.
\newblock Augmented distributed gradient methods for multi-agent optimization under uncoordinated constant stepsizes.
\newblock In {\em Proceedings of 54th IEEE Conference on Decision and Control (CDC)}, pages 2055--2060, 2015.

\bibitem{YYH+21}
Bicheng Ying, Kun Yuan, Hanbin Hu, Yiming Chen, and Wotao Yin.
\newblock {BlueFog: Make decentralized algorithms practical for optimization and deep learning}.
\newblock {\em arXiv e-prints}, arXiv:2111.04287, 2021.

\bibitem{YAYS20}
Kun Yuan, Sulaiman~A. Alghunaim, Bicheng Ying, and Ali~H. Sayed.
\newblock On the influence of bias-correction on distributed stochastic optimization.
\newblock {\em IEEE Transactions on Signal Processing}, 68:4352--4367, 2020.

\bibitem{YLY16}
Kun Yuan, Qing Ling, and Wotao Yin.
\newblock On the convergence of decentralized gradient descent.
\newblock {\em SIAM Journal on Optimization}, 26(3):1835--1854, 2016.

\bibitem{YYZS18}
Kun Yuan, Bicheng Ying, Xiaochuan Zhao, and Ali~H. Sayed.
\newblock {Exact diffusion for distributed optimization and learning—Part I: Algorithm development}.
\newblock {\em IEEE Transactions on Signal Processing}, 67(3):708--723, 2018.

\end{thebibliography}

\end{document}